\documentclass[11pt]{article}
\usepackage{amsmath,amsthm,amsfonts,amssymb}
\usepackage[all]{xy}
\usepackage[usenames]{color} 
\usepackage{pdfsync}
\usepackage{tensor}

\newcounter{enumitemp}

\newcommand\pref[1]{(\ref{#1})}

\newtheorem{theorem}{Theorem}
\newtheorem{lemma}[theorem]{Lemma}
\newtheorem{corollary}[theorem]{Corollary}
\newtheorem{proposition}[theorem]{Proposition}

\theoremstyle{definition}

\newtheorem{definition}[theorem]{Definition}

\newcounter{remarks}
{\paragraph*{Remarks}\smallskip
 \begin{list}{\arabic{remarks}. }{\usecounter{remarks}%
 \setlength{\leftmargin}{0in}%
 \setlength{\rightmargin}{0in}%
 \setlength{\labelsep}{0pt}%
 \setlength{\labelwidth}{0pt}%
 \setlength{\listparindent}{0pt}%
 }
}
{
\end{list}
}

\newcommand\from\colon
\newcommand\inv{{-1}}
\newcommand\subgroup{<}
\newcommand\infinity\infty

\newcommand\disjunion\coprod
\newcommand\act\curvearrowright

\DeclareMathOperator\image{image}

\DeclareMathOperator\coindex{coindex}

\DeclareMathOperator\Area{Area}

\newcommand{\N}{{\mathbb N}}

\newcommand{\I}{{\mathcal I}}
\newcommand\M{{\mathcal M}}
\newcommand{\V}{\mathcal V}

\renewcommand\O{{\mathcal O}}

\newcommand{\Out}{\mathsf{Out}}
\newcommand{\Aut}{\mathsf{Aut}}
\newcommand{\Inn}{\mathsf{Inn}}

\newcommand{\Stab}{\mathsf{Stab}}

\newcommand{\F}{\mathcal F}

\newcommand{\A}{\mathcal A}

\newcommand{\ti} {\tilde}

\newcommand{\noneg}{NEG}
\renewcommand\neg\noneg

\newcommand{\wt}{\widetilde}

\newcommand\nt{\text{nt}}

\newcommand{\comment}[1]{}

\newcommand\BookOne{\cite{BFH:TitsOne}}

\newcommand\bdy\partial

\newcommand\intersect\cap
\newcommand\union\cup
\newcommand\<\langle
\renewcommand\>\rangle
\newcommand\meet\wedge
\newcommand\composed{\circ}
\newcommand\cross\times
\newcommand\restrict{\bigm |}
\newcommand\wh{\widehat}
\newcommand\inject\hookrightarrow
\newcommand\reals{\mathbf{R}}

\newcommand\abs[1]{\left|#1\right|}

\DeclareMathOperator\rank{rank}

\DeclareMathOperator\CVK{\mathcal{K}}
 
 \newcommand\surjection\twoheadrightarrow
 
\newcommand\suchthat{\bigm|}

\DeclareMathOperator\MCG{\mathcal{MCG}}

\newcommand\MilnorSvarc{Milnor-\v{S}varc}

\newcommand\Autre[1][]{\mathbb{A}^{#1}}

\newcommand\con{\gamma}

\title{Lipschitz retraction and distortion for subgroups of $\Out(F_n)$}

\author{Michael Handel and Lee Mosher}

\begin{document}

\maketitle

\begin{abstract}
Given a free factor $A$ of the rank n free group $F_n$, we characterize when the subgroup of $\Out(F_n)$ that stabilizes the conjugacy class of $A$ is distorted in $\Out(F_n)$. We also prove that the image of the natural embedding of $\Aut(F_{n-1})$ in $\Aut(F_n)$ is nondistorted, that the stabilizer in $\Out(F_n)$ of the conjugacy class of any free splitting of $F_n$ is nondistorted, and we characterize when the stabilizer of the conjugacy class of an arbitrary free factor system of $F_n$ is distorted. In all proofs of nondistortion, we prove the stronger statement that the subgroup in question is a Lipschitz retract. As applications we determine Dehn functions and automaticity for $\Out(F_n)$ and $\Aut(F_n)$.
\end{abstract}

Given a finitely generated group $G$, a finitely generated subgroup $H$ is \emph{undistorted} in $G$ if the inclusion map $H \inject G$ is a quasi-isometric embedding with respect to word metrics. Nondistortion may be verified by constructing a Lipschitz retraction $G \mapsto H$ (see Lemma~\ref{LemmaCoarseRetract}), and this is particularly useful because it has extra consequences: if $G$ is finitely presented, and if there is a Lipschitz retraction $G \to H$, then $H$ is finitely presented and the Dehn function of $G$ is dominated below by the Dehn function of $H$ (see Proposition~\ref{PropDehnBound}).

In the mapping class group $\MCG(S)$ of any closed oriented surface $S$, the subgroup that stabilizes the isotopy class of any essential curve or subsurface of $S$ is a Lipschitz retract of $\MCG(S)$, and therefore undistorted. The proof of this fact is a simple application of the projection methods of \cite{MasurMinsky:complex2}; see ``Motivation from mapping class groups'' below.

As a step in understanding the large scale geometry of the outer automorphism group $\Out(F_n)$ of a rank $n$ free group $F_n$, one might consider the following. The action of the automorphism group $\Aut(F_n)$ on subgroups of $F_n$ induces an action of $\Out(F_n)$ on conjugacy classes of subgroups of $F_n$. Given a free factor $A < F_n$, the stabilizer $\Stab[A]$ of its conjugacy class is a finitely generated subgroup of  $\Out(F_n)$. Regarding a free factor as an analogue of an essential subsurface, one might ask whether $\Stab[A]$ is an undistorted subgroup of $\Out(F_n)$. The answer is surprising:

\begin{theorem}\label{TheoremCorank}
Given a proper, nontrivial free factor $A < F_n$:
\begin{enumerate}
\item \label{ItemCorank1Undistorted}
If $\rank(A) = n-1$ then $\Stab[A]$ is a Lipschitz retract of $\Out(F_n)$, and therefore is undistorted in $\Out(F_n)$.
\item \label{ItemHigherCorankDistorted}
If $\rank(A) \le n-2$ then $\Stab[A]$ is distorted in $\Out(F_n)$.
\end{enumerate}
\end{theorem}

\break

Upon discovering this theorem, we searched for and found wider classes of Lipschitz retract subgroups and distorted subgroups, and we applied Lipschitz retracts to settle questions about Dehn functions and automaticity of $\Out(F_n)$ and $\Aut(F_n)$: 
\begin{itemize}
\item[$*$] Theorems~\ref{CorOutAutRetract} and~\ref{TheoremAutRetract} assert that there are subgroups isomorphic to $\Aut(F_{n-1})$ in both of the groups $\Out(F_n)$ and $\Aut(F_n)$ which are Lipschitz retracts and therefore nondistorted.
\item[$*$] Corollaries~\ref{CorollaryDehn} and~\ref{CorollaryNotAutomatic} say that for $n \ge 3$ the groups $\Aut(F_n)$ and $\Out(F_n)$ have exponential Dehn functions and are not automatic.
\item[$*$] Theorem~\ref{TheoremFreeSplitting} generalizes case~\pref{ItemCorank1Undistorted} of Theorem~\ref{TheoremCorank} by providing Lipschitz retractions to stabilizers of arbitrary free splittings of $F_n$, which are therefore undistorted in $\Out(F_n)$.
\item[$*$] Theorem~\ref{TheoremCoindex} generalizes Theorem~\ref{TheoremCorank} by characterizing Lipschitz retraction and distortion for stabilizers of conjugacy classes of arbitrary free factor systems of $F_n$.
\end{itemize}

\subsection*{More Lipschitz retracts, and applications.} In addition to the special case of Theorem~\ref{TheoremFreeSplitting} mentioned just above, here is another important special case:

\begin{corollary} \label{CorOutAutRetract}
There is a Lipschitz retraction from $\Out(F_n)$ to a subgroup isomorphic to $\Aut(F_{n-1})$, which is therefore undistorted.
\end{corollary}
\noindent The proof is given below, as an application of Theorem~\ref{TheoremFreeSplitting}.

Around 2002 Jose Burillo asked whether the natural $\Aut(F_{n-1})$ subgroup of $\Aut(F_n)$ is undistorted \cite{Burillo:private}. This subgroup is defined as all automorphisms of $F_n = \<a_1,\ldots,a_{n-1},a_n\>$ which preserve the free factor $\<a_1,\ldots,a_{n-1}\>$ and the element $a_n$. The methods of Theorem~\ref{TheoremCorank}~\pref{ItemCorank1Undistorted} apply to answer this question affirmatively, although the proof, given in outline below, is even easier in this case; see Section~\ref{SectionAutNondistortion} for full details.

\begin{theorem}\label{TheoremAutRetract}
There is a Lipschitz retraction from $\Aut(F_n)$ to the natural subgroup isomorphic to $\Aut(F_{n-1})$, which is therefore undistorted. 
\end{theorem}

Combining Theorem~\ref{CorOutAutRetract}, Theorem~\ref{TheoremAutRetract}, and an evident induction we obtain:

\begin{corollary}\label{CorollaryOutToAutRetract}
If $n \ge 4$ there are Lipschitz retractions from each of the groups $\Out(F_n)$ and $\Aut(F_n)$ to subgroups isomorphic to $\Aut(F_3)$. 
\end{corollary}

In \cite{HatcherVogtmann:isoperimetric}, Hatcher and Vogtmann found exponential upper bounds for Dehn functions of $\Aut(F_n)$ and $\Out(F_n)$ when $n \ge 3$. In \cite{BridsonVogtmann:GeometryOfAut}, Bridson and Vogtmann found exponential lower bounds when $n=3$, thereby completely determining the growth types of those Dehn functions to be exponential. Lower bounds in rank~$\ge 4$ were left unresolved. 

Combining Corollary~\ref{CorollaryOutToAutRetract} with the lower bounds in rank~$3$ from \cite{BridsonVogtmann:GeometryOfAut}, and with the fact that the Dehn function of a group is dominated below by the Dehn function of any Lipschitz retract (see Proposition~\ref{PropDehnBound}), we obtain the desired exponential lower bounds for $n \ge 4$, thus completely determining the growth types of the Dehn functions of $\Aut(F_n)$ and $\Out(F_n)$:

\begin{corollary}\label{CorollaryDehn}
For all $n \ge 4$, the Dehn functions of the groups $\Out(F_n)$ and $\Aut(F_n)$ have exponential lower bounds, and therefore they have exponential growth type.
\end{corollary}

\noindent After this result and its proof were explained to participants of the conference in Luminy in June 2010, Bridson and Vogtmann were able to give another proof of Corollary~\ref{CorollaryDehn}; see \cite{BridsonVogtmann:DehnFunctions}.

\bigskip

In \cite{BridsonVogtmann:GeometryOfAut}, Bridson and Vogtmann use their exponential lower bound for Dehn functions to prove that $\Aut(F_3)$ and $\Out(F_3)$ are not automatic, because the Dehn function of every automatic group has a quadratic upper bound \cite{ECHLPT}. They also proved that $\Aut(F_n)$ and $\Out(F_n)$ are not biautomatic for all $n \ge 4$, but automaticity was left unresolved in those cases. Using Corollary~\ref{CorollaryDehn} we can now resolve this issue:

\begin{corollary}\label{CorollaryNotAutomatic}
For all $n \ge 4$, the groups $\Out(F_n)$ and $\Aut(F_n)$ are not automatic.
\end{corollary}

\subparagraph{Remark.} The proof in \cite{BridsonVogtmann:GeometryOfAut} that $\Out(F_n)$ and $\Aut(F_n)$ are not biautomatic is inductive, based on exploiting theorems about centralizers of finite sets in biautomatic groups due to Gersten and Short \cite{GerstenShort:rational}. It seems possible to us that the method of Lipschitz retraction to subgroups may prove effective in other contexts similar to the present one for inductive verification of lower bounds on isoperimetric functions and failure of automaticity.

\subsection*{Stabilizers of free splittings.} The two special cases of Lipschitz retracts of $\Out(F_n)$ mentioned so far --- of Theorem~\ref{TheoremCorank} case~\pref{ItemCorank1Undistorted};  and Corollary~\ref{CorOutAutRetract} ---  are each best expressed in the language of trees. This leads to a rich collection of undistorted subgroups of $\Out(F_n)$, as expressed in our next theorem. 

Define a \emph{free splitting} of $F_n$ to be a minimal action of $F_n$ on a nontrivial, simplicial tree $T$ with trivial edge stabilizers. As is traditional in this subject, the notation $T$ will incorporate both the tree and the action of $F_n$; we shall denote the action of $g \in F_n$ on $T$ by $x \mapsto g \cdot x$. Two free splittings $T,T'$ are \emph{conjugate} if there exists a simplicial isomorphism $f \from T \mapsto T'$ such that $f(g \cdot x) = g \cdot f(x)$. The group $\Out(F_n)$ acts on the set of conjugacy classes of free splittings of $F_n$, and the stabilizer of the conjugacy class of a free splitting $T$ is denoted~$\Stab[T]$. A free splitting may also be understood using Bass-Serre theory, being the Bass-Serre tree of its quotient graph of groups.

\begin{theorem}\label{TheoremFreeSplitting}
Given a free splitting $T$ of $F_n$, there is a Lipschitz retraction from $\Out(F_n)$ to its subgroup $\Stab[T]$, which is therefore undistorted in $\Out(F_n)$.
\end{theorem}

Here, briefly, are applications of this theorem to the two special cases mentioned above.

\begin{proof}[Sketch of proof of Theorem~\ref{TheoremCorank} case~\pref{ItemCorank1Undistorted}] Let $T$ be the free splitting of $F_n = \<a_1,\ldots,a_n\>$ which is the Bass-Serre tree of the graph of groups $X$ having one vertex with group label $F_n = \<a_1,\ldots,a_{n-1}\>$ and having one loop edge with trivial group label and stable letter~$a_n$. The Serre fundamental group of $X$ \cite{Serre:trees} is identified with $F_n$. In Section~\ref{SectionCorankOne} we will prove the equation $\Stab[F_{n-1}] = \Stab[T]$. Combined with Theorem~\ref{TheoremFreeSplitting} it follows that $\Stab[F_{n-1}]$ is a coarse retract of $\Out(F_n)$ and is therefore undistorted. 
\end{proof}
\noindent
See Section~\ref{SectionCorankOne} for more details of this proof in a broader setting, and see below just after the statement of Theorem~\ref{TheoremCoindex} for a discussion of this broader setting in terms of ``co-edge~$1$ free splittings'' and ``co-index~$1$ free factor systems''.

\begin{proof}[Proof of Corollary~\ref{CorOutAutRetract}] In Lemma~\ref{LemmaFSDescription} we give a formula for the algebraic structure of a certain finite index subgroup of $\Stab[T]$ for any free splitting $T$, expressed in terms of the structure of the quotient graph of groups $X=T/F_n$. As a special case of Lemma~\ref{LemmaFSDescription}, if $X$ has a unique vertex $V$ labelled by a nontrivial free group $A$, and if $V$ has valence~$1$, then $\Stab[T]$ has a finite index subgroup isomorphic to $\Aut(A)$. Combining this with Theorem~\ref{TheoremFreeSplitting} and the fact that any subgroup Lipschitz retracts to any of its finite index subgroups, in order to find a Lipschitz retract subgroup of $\Out(F_n)$ isomorphic to $\Aut(F_{n-1})$ it suffices by Bass-Serre theory to construct a graph of groups presentation $X$ for $F_n$ having one valence~$1$ vertex labelled with the group $F_{n-1}$, and all other vertices and edges labelled by the trivial group, so that the underlying graph of $X$ has rank~$1$. The desired $X$ is a ``sewing needle'' with two vertices $V,W$ and two edges $A,B$, the edge $A$ having one end on $V$ and one on $W$, the edge $B$ having both ends on $W$, and the vertex labelled by~$F_{n-1}$ being~$V$.
\end{proof}

\subsection*{Proofs of Lipschitz retraction}

Section~\ref{SectionTreeStabilizers} contains the general proof of Theorem~\ref{TheoremFreeSplitting}. Here is an outline of the special case from Theorem~\ref{TheoremCorank}~\pref{ItemCorank1Undistorted}, of a tree $T$ whose quotient graph of groups is a loop with one vertex labelled by the free factor $A = F_{n-1}$ and one trivially labelled edge. The proof takes place in $\CVK_n$, the spine of the outer space of $F_n$, a locally finite, contractible simplicial complex on which $\Out(F_n)$ acts properly and cocompactly. The vertices of $\CVK_n$ are the free splittings $S$ such that the action of $F_n$ on $S$ is proper with every vertex of valence~$\ge 3$. There is a connected subcomplex $\CVK^T_n$ whose vertices consist of all such $S$ with the additional property that the minimal $A$-invariant subtree $S^A \subset S$ is disjoint from all of its translates by elements of $F_n-A$ and $S$ has a unique edge orbit not contained in~$S^A$. The group $\Stab[T]$ acts properly and cocompactly on the subcomplex $\CVK^T_n$. Existence of a Lipschitz retract $\Out(F_n) \mapsto \Stab[T]$ is therefore equivalent to existence of a Lipschitz retract from the vertex set of $\CVK_n$ to the vertex set of its subcomplex $\CVK^T_n$; see Corollary~\ref{CorollarySubgroupComplex} for details of this equivalence. Given a vertex $S' \in \CVK_n$, our task is therefore to construct a vertex of $\CVK^T_n$ denoted $S$. The given $S'$ certainly has a well-defined minimal $A$-invariant subtree $S^A$, but that subtree might not be disjoint from all its translates by elements of $F_n-A$. This problem is easily remedied: pull all of these translates apart, forcing them to be pairwise disjoint. One obtains a forest on which $F_n$ acts properly and cocompactly, whose components are in one-to-one correspondence with the conjugates of $A$ in $F_n$. The tricky part of the proof is to attach an orbit of edges to make this forest into the desired tree $ S$, and to do this in a manner which is quasi-isometrically natural, in order that the map of $0$-skeleta taking the free splitting $S'$ to the free splitting $S$ be a Lipschitz retract as desired. At the end of Section~\ref{SectionTreeStabilizers} we give a \emph{Remark} explaining why some seemingly natural schemes for attaching the edges fail to satisfy the Lipschitz requirement.

\smallskip

Section~\ref{SectionAutNondistortion} contains the detailed proof of Theorem~\ref{TheoremAutRetract}. The proof takes place in the spine $\Autre_n$ of the ``autre espace'' of $F_n$ \cite{HatcherVogtmann:Cerf}. Although $\Autre_n$ has a ``tree'' model like the one we use for the spine of outer space, and while the proof of Theorem~\ref{TheoremAutRetract} can be couched in terms of the tree model, it turns out that delicate ``edge attachment'' problem which arises in the proof of Theorem~\ref{TheoremFreeSplitting} and whose solution is most propitiously expressed in tree language is not an issue in the proof of Theorem~\ref{TheoremAutRetract}. The most economical proof of the latter uses the model for $\Autre_n$ couched in terms of pointed marked graphs. The spine $\Autre_n$ is a locally finite, contractible simplicial complex on which $\Aut(F_n)$ acts properly and cocompactly: a~$0$-simplex is a connected, pointed graph with no vertex of valence~$1$ whose fundamental group is identified with $F_n$, and a $1$-simplex connects one $0$-simplex to a second one if the second graph is obtained from the first by collapsing a subforest. The complex $\Autre_n$ has connected subcomplex on which $\Aut(F_{n-1})$ acts properly and cocompactly. Given an arbitrary pointed $F_n$-marked graph $G$ representing a vertex of $\Autre_n$, we construct a vertex in its $\Aut(F_{n-1})$ subcomplex. There are three steps. First, in the pointed universal cover $\wt G$, take the $F_{n-1}$ minimal subtree $S$ with base point being the point closest to the base point of $\wt G$. Next, take the quotient of $S$ to get a pointed $F_{n-1}$ marked graph. Finally, attach a loop to the new base point. The resulting pointed $F_n$ marked graph is in the $\Aut(F_{n-1})$ subcomplex and the map defined is the desired Lipschitz retraction.

\subparagraph{Motivation from mapping class groups.} Consider a finite type, oriented surface $S$ with mapping class group $\MCG(S)$. Given an essential simple closed curve $c \subset S$ with isotopy class $[c]$, its stabilizer subgroup $\Stab[c] \subgroup \MCG(S)$ is undistorted. This fact is an immediate corollary of the Masur-Minsky summation, Theorem~6.12 of \cite{MasurMinsky:complex2}, and has also been proved by Hamenstadt in \cite{Hamenstadt:MCGII:quasigeodesics} using train track methods. However, there is a simple proof of the stronger fact that $\Stab[c]$ is a Lipschitz retract of $\MCG(S)$. This proof, which is not in the literature, uses only the most elementary methods concerning markings and projections taken from the early sections of \cite{MasurMinsky:complex2}. Here is a very brief sketch. In place of the spine of outer space one uses the \emph{marking complex} $\M_S$, whose vertices are ``markings'' $\mu$ of $S$, meaning pants decompositions equipped with a transverse curve on each pants curve; edges of $\M_S$ corresponding to ``elementary moves'' between markings. In $\M_S$ there is a subcomplex $\M_S^{[c]}$ consisting of markings whose pants decomposition contains the curve $c$. The action $\MCG(S) \act \M_S$ is properly discontinuous and cocompact and $\M_S^{[c]}$ is invariant and cocompact under the restricted action of $\Stab[c]$, and so by Corollary~\ref{CorollarySubgroupComplex} it suffices to construct a Lipschitz retraction $\M_S \mapsto \M_S^{[c]}$. An arbitrary marking $\mu$ on $S$ projects to a marking in the subcomplex $\M_S^{[c]}$ by surgering $\mu$ along intersections between $\mu$ and $c$ to get another marking that contains $c$. As in \cite{MasurMinsky:complex2} it is easy to check that two markings that differ by an elementary move have uniformly close images under this projection.

\subsection*{Stabilizers of free factor systems, and proofs of distortion.} We extend Theorem~\ref{TheoremCorank} as follows. A set $\F=\{[A_1],\ldots,[A_K]\}$ of conjugacy classes of nontrivial free factors of $F_n$ is called a \emph{free factor system} if it is represented by subgroups $A_1,\ldots,A_K$ such that there exists a free factorization of the form $F_n = A_1 * \ldots * A_K * B$, where $B$ may be trivial. The action of $\Out(F_n)$ on conjugacy classes of free factors extends to an action on free factor systems. The \emph{co-index} of $\F = \{[A_1],\ldots,[A_K]\}$ is the number
$$\coindex(\F) = \abs{\chi(F_n)} - \sum_{k=1}^K \abs{\chi(A_k)} = (n-1) - \sum_{k=1}^K (\rank(A_k)-1)
$$
Topologically this is the number of edges that must be attached to the vertices of the union of $K$ connected graphs of ranks $\rank(A_1),\ldots,\rank(A_K)$ to make a connected graph of rank~$n$. It is also the minimum, among all graphs of rank $n$ having a subgraph with $K$ components of ranks equal to $\rank(A_1),\ldots,\rank(A_K)$, of the number of edges not in the subgraph.

\begin{theorem}\label{TheoremCoindex}
Given a free factor system $\F$ of $F_n$,
\begin{enumerate}
\item \label{ItemFFSCoindexOne}
If $\coindex(\F)=1$ then $\Stab(\F)$ is a Lipschitz retract of $\Out(F_n)$, and therefore is undistorted in $\Out(F_n)$.
\item \label{ItemFFSCoindexHigher}
If $\coindex(\F) \ge 2$ then $\Stab(\F)$ is distorted in $\Out(F_n)$.
\end{enumerate}
\end{theorem}

Section~\ref{SectionCorankOne} contains the proof of Theorem~\ref{TheoremCoindex}~\pref{ItemFFSCoindexOne}, by showing that for any free factor system $\F$ of coindex~1 there is a free splitting $T$ with one edge orbit such that $\Stab(\F)=\Stab[T]$, and applying a result from \BookOne, Corollary~3.3.2, concerning homotopy equivalences of marked graphs that preserve a subgraph consisting of all edges but one.

Section~\ref{SectionCorank2} contains the proof of Theorem~\ref{TheoremCoindex}~\pref{ItemFFSCoindexHigher}. Here is an outline for the special case considered in Theorem~\ref{TheoremCorank}~\pref{ItemHigherCorankDistorted}: given a proper, nontrivial free factor $A \subgroup F_n$ of rank $r \le n-2$, we shall prove distortion of $\Stab[A]$ in $\Out(F_n)$. To do this we produce a sequence $\phi_k \in \Stab[A]$ whose word length in $\Out(F_n)$ has a linear upper bound in $k$ but whose word length in $\Stab[A]$ has an exponential lower bound in $k$. The sequence $\phi_k$ is described as follows. Let $R_n$ be the rose with $n$ oriented edges labelled $e_1,\ldots,e_n$, whose path homotopy classes rel basepoint give a free basis of $F_n$ also denoted $e_1,\ldots,e_n$. We may assume that $A = \<e_1,\ldots,e_r\>$ with $1 \le r \le n-2$. Let $\theta \in \Out(F_n)$ be defined by a homotopy equivalence $\Theta \from R_n \to R_n$ which is the identity on $e_{r+2} \union\cdots\union e_n$, and whose restriction to $e_1 \union\cdots\union e_{r+1}$ is an irreducible train track map of exponential growth. By the train track property, the path $u_k = \Theta^k(e_1)$ has no cancellation and so may be regarded as a reduced word in the generators $e_1,\ldots,e_{r+1}$. Now define $\phi_k \in \Stab[A]$ by the automorphism 
$$\begin{cases}
e_i &\mapsto \,\, e_i  \qquad\qquad \text{if $i < n$} \\
e_{n} &\mapsto \,\, e_{n} u_k
\end{cases}
$$ 
From the expression $\phi_k = \theta^k \phi_0 \theta^{-k}$ one sees immediately that the word length of $\phi_k$ in $\Out(F_n)$ has a linear upper bound in~$k$. The hard work is to prove that the word length of $\phi_k$ in $\Stab[A]$ has an exponential lower bound. The key properties of $u_k$ are that it is a reduced word in the letters of the rank $r+1$ free factor $B = \<e_1,\ldots,e_r,e_{r+1}\>$, and that the number of occurrences of the letter $e_{r+1}$ grows exponentially in $k$. 

The technique at the heart of the proof is counting occurrences of $e_{r+1}$. Fixing any conjugacy class $c$ in $\Out(F_n)$, for any $\psi \in \Stab[A]$ we will describe a way to count occurrences of the letter $e_{r+1}$ in the conjugacy class $\psi(c)$, and we will show that this count defines a Lipschitz function on $\Stab[A]$ with respect to word metric. Applying this to $c=[e_{r+2}]$ and to any factorization of $\psi=\phi_k$ as a word in the generators of $\Stab[A]$, it will follow that the word length of $\phi_k$ in $\Stab[A]$ is bounded below by a constant multiple of the number of occurrences of $e_{r+1}$ in $u_k$, which has an exponential lower bound. Section~\ref{SectionCorankOne} will contain a result about corank~$1$ free factors which is needed to prove the Lipschitz bound, and which is applied to the inclusion of $A$ as a corank~$1$ free factor of~$B$.

\subsection*{Contents.}

Section~\ref{SectionPrelim} contains preliminary results about quasi-isometric geometry, and about $\Out(F_n)$ and $\Aut(F_n)$, outer space and autre espace, and the spines thereof. In particular, Corollary~\ref{CorollarySubgroupComplex} sets the stage for later proofs of coarse Lipschitz retraction or distortion of the subgroups of any group, by showing that these are equivalent to coarse Lipschitz retraction or distortion of corresponding subcomplexes, via the \MilnorSvarc\ lemma.

Proofs regarding subgroups of $\Out(F_n)$ and $\Aut(F_n)$ that are coarse Lipschitz retracts are contained in Sections~\ref{SectionAutNondistortion}, \ref{SectionTreeStabilizers}, and~\ref{SectionCorankOne}. Proofs regarding distorted subgroups are contained in Section~\ref{SectionCorank2}, with supporting material in Section~\ref{SectionFFSubcomplex}. The distortion proofs and the nondistortion proofs are almost entirely independent, and one can choose to read one or the other first, with overlapping material mostly isolated in Section~\ref{SectionPrelim}. In particular, readers interested in distortion may, after reviewing relevant material in Section~\ref{SectionPrelim}, skip straight to Section~\ref{SectionFFSubcomplex} or even to Section~\ref{SectionCorank2}.

Section~\ref{SectionAutNondistortion} contains the proof of Theorem~\ref{TheoremAutRetract}, the existence of a coarse Lipschitz retraction $\Aut(F_n) \mapsto \Aut(F_{n-1})$, which is given first because of its more elementary nature. 

Section~\ref{SectionTreeStabilizers} contains the proof of Theorem~\ref{TheoremFreeSplitting}, the existence of a coarse Lipschitz retraction $\Out(F_n) \mapsto \Stab[T]$ for any free splitting $F_n \act T$. Section~\ref{SectionFreeSplittingSubcomplex} describes a subcomplex $\CVK^T_n \subset \CVK_n$ corresponding to $\Stab[T]$, and contains the proof of Lemma~\ref{LemmaFSDescription} describing the algebraic structure of (a finite index subgroup of) $\Stab[T]$. Section~\ref{SectionSubcomplexRetract} describes a coarse Lipschitz retraction $\CVK_n \mapsto \CVK^T_n$.

Section~\ref{SectionTandA} contains the proof of Theorem~\ref{TheoremCoindex}, that for any free factor system $\F$ of $F_n$ the subgroup $\Stab(\F)$ is a Lipschitz retract of $\Out(F_n)$ if $\coindex(\F)=1$, and is distorted if $\coindex(\F) \ge 2$. Subsection~\ref{SectionCorankOne} handles the case where $\coindex(\F)=1$, by an application of Theorem~\ref{TheoremFreeSplitting}. Subsection~\ref{SectionFFSubcomplex} describes a subcomplex $\CVK^\F_n \subset \CVK_n$ corresponding to $\Stab(\F)$.  Subsection~\ref{SectionCorank2} handles the case where $\coindex(\F) \ge 2$ by proving that the subcomplex $\CVK^\F_n$ is distorted in $\CVK_n$.

\subparagraph{Acknowledgements.} We thank Kevin Whyte for pointing out the need to check finiteness of $i_{A,B}(c,T)$ in the proof of Lemma~\ref{LemmaCount}.

The second author would like to thank his brother-in-law John Garnett for very propitiously running out of gas and walking into town to get a fill, leaving him with the car to ponder co-edge attachments.

We would also like to thank the organizers of the June 2010 Luminy conference in honor of Karen Vogtmann, where these results were first presented.

\section{Preliminaries}
\label{SectionPrelim}

\subsection{Quasi-isometries and Lipschitz retracts}
\label{SectionQuasis}

Given two metric spaces $X,Y$, a function $f \from X \to Y$, and constants $K \ge 1$, $C \ge 0$, we say that $f$ is \emph{$K,C$ coarse Lipschitz} if $d_Y(f(a),f(b)) \le K d_X(a,b) + C$ for all $a,b \in X$, we say that $f$ is a \emph{$K,C$ quasi-isometric embedding} if it is $K,C$ coarse Lipschitz and $\frac{1}{K} d_X(a,b) - C \le d_Y(f(a),f(b))$ for all $a,b \in X$, and we say that $f$ is a \emph{$K,C$ quasi-isometry} if it is a $K,C$ quasi-isometric embedding and for each $y \in Y$ there exists $x \in X$ such that $d_Y(f(x),y) \le C$. For each of these terms we often drop the reference to the constants $K,C$. A composition of quasi-isometries is a quasi-isometry, and any quasi-isometry $f \from X \to Y$ has a \emph{coarse inverse}, a map $\bar f \from Y \to X$ such that each of the maps $\bar f \composed f$ and $f \composed \bar f$ move points a uniformly bounded distance; the quasi-isometry constants of $\bar f$ and the distance bounds for $\bar f \composed f$ and $f \composed \bar f$ depend only on the quasi-isometry constants of $f$. 

A metric space $X$ is \emph{proper} if every closed ball is compact, and $X$ is \emph{geodesic} if for any $a,b \in X$ there is a rectifiable path from $a$ to $b$ whose length equals $d_X(a,b)$. For example (see Section~\ref{SectionGraphs}), every simplicial complex supports a geodesic metric in which each simplex is isometric to a Euclidean simplex of side length~$1$; if the complex is locally finite then this metric is proper. Any action $G \act X$ on a simplicial complex which respects simplicial coordinates is an isometry.

\begin{lemma}[\MilnorSvarc\ \cite{Svarc:VolumeInvariant, Milnor:curvature}]
For any group $G$ and any proper, geodesic metric space $X$, if there exists a  properly discontinuous, cocompact, isometric action $G \act X$ then $G$ is finitely generated. Furthermore, for any such action and any point $x \in X$, the orbit map $g \mapsto g \cdot x$ is a quasi-isometry $\O \from G \to X$, where $G$ is equipped with the word metric of any finite generating set.
\end{lemma}

Given a geodesic metric space $X$, a subspace $Y \subset X$ is said to be \emph{rectifiable} if any two points $x,y \in Y$ are endpoints of some path in $Y$ which is rectifiable in $X$ and whose length is minimal amongst all rectifiable paths in $Y$ with endpoints $x,y$. It follows that minimal path length defines a geodesic metric on $Y$. Examples include any connected subcomplex of any connected simplicial complex, using a natural simplicial metric where each simplex is isometric to a Euclidean simple of side length~$1$. We say that $Y$ is a \emph{(coarse) Lipschitz restract} of $X$ if there exists a \emph{(coarse) Lipschitz retraction} $f \from X \mapsto Y$, meaning a (coarse) Lipschitz function that restricts to the identity on $Y$. 

\begin{corollary}\label{CorollarySubgroupComplex}
If $G$ is a finitely generated group acting properly discontinuously and cocompactly by isometries on a connected locally finite simplicial complex $X$, if $H \subgroup G$ is a subgroup, and if $Y \subset X$ is a nonempty connected subcomplex which is $H$-invariant and $H$-cocompact, then:
\begin{enumerate}
\item \label{ItemFG}
$H$ is finitely generated, 
\item \label{ItemUndistortedInclusion}
$H$ is undistorted in $G$ if and only if the inclusion $Y \inject X$ is a quasi-isometric embedding.
\item \label{ItemLipschitzRetract} The following are equivalent:
\begin{enumerate}
\item \label{ItemGroupLipschitzRetract} $H$ is a Lipschitz retract of $G$.
\item \label{ItemZeroLipschitzRetract} The $0$-skeleton of $Y$ is a Lipschitz retract of the $0$-skeleton of $X$.
\item \label{ItemOneLipschitzRetract} The $1$-skeleton of $Y$ is a Lipschitz retract of the $1$-skeleton of $X$. 
\item \label{ItemCoarseLipschitzRetract} $Y$ is a coarse Lipschitz retract of $X$.
\end{enumerate}
\end{enumerate}
\end{corollary} 

\begin{proof} 
The actions of $G$ on $X$ and $H$ on $Y$ satisfy the hypotheses of the \MilnorSvarc\ lemma so \pref{ItemFG} follows. Also, choosing a base vertex $x \in Y$, it follows that the orbit map $g \mapsto g \cdot x$ defines quasi-isometries $\O_G \from G \mapsto X$ and $\O_H \from H \mapsto Y$. Denoting inclusions as $i_Y \from Y \inject X$ and $i_H \from H \inject G$ we have $\O_G \composed i_H = i_Y \composed \O_H$.  It follows that $i_Y$ is a quasi-isometric embedding if and only if $i_H$ is, proving~\pref{ItemUndistortedInclusion}. Choose $C \ge 0$ so that each point of $X$ is within distance $C$ of $\image(\O_G)$ and each point of $Y$ is within distance $C$ of $\image(\O_H)$.

To prove \pref{ItemGroupLipschitzRetract}$\implies$\pref{ItemZeroLipschitzRetract}, suppose $\pi_H \from G \to H$ is a Lipschitz retraction. Define $\pi_Y$ from the $0$-skeleton of $X$ to the $0$-skeleton of $Y$ as follows: given a vertex $v \in X$, choose $g_v \in G$ so that $\O_G(g_v) = g_v \cdot x$ is within distance $C$ of $v$, and define $\pi_Y(v)=\O_H \composed \pi_H(g_v)$. Chasing through diagrams, $\pi_Y$ is easily seen to be Lipschitz, although it might not be a retraction. But if $v \in Y$ then we can choose $g_v$ to be in $H$ and so $\pi_Y(v) = \O_H(\pi_H(g_v)) = \O_H(g_v) = g_v \cdot x$ is within distance $C$ of $v$ and so, by moving $\pi_Y(v)$ a uniformly bounded distance for all vertices $v$ of $Y$, we obtain a Lipschitz retraction.

The implication \pref{ItemZeroLipschitzRetract}$\implies$\pref{ItemOneLipschitzRetract} follows by extending any Lipschitz retract of $0$-skeleta continuously so as to map each edge to an edge path of bounded length. To prove the implication \pref{ItemOneLipschitzRetract}$\implies$\pref{ItemCoarseLipschitzRetract}, extend across each higher dimensional simplex $\sigma$ by the identity on the interior of $\sigma$ if $\sigma \subset Y$, and otherwise extend so that the image of the interior of $\sigma$ is contained in the image of the $1$-skeleton of $\sigma$; continuity is ignored in this extension, hence one obtains only a \emph{coarse} Lipschitz retract.

To prove \pref{ItemCoarseLipschitzRetract}$\implies$\pref{ItemGroupLipschitzRetract}, suppose $\pi_Y \from X \to Y$ is a coarse Lipschitz retract. For each $g \in G$, choose $h = \pi_H(g) \in H$ so that $\pi_Y(\O_G(g))$ is within distance $C$ of $\O_H(h)$. Chasing through diagrams $\pi_H$ is again easily seen to be Lipschitz. It might not be a retraction, but if $g \in H$ then $\O_G(g) = \O_H(g) \in Y$ is equal to $\pi_Y(\O_H(g))$ and is within $C$ of $\O_H(h)$, and so $d_H(g,h)$ is uniformly bounded; by moving $h = \pi_H(g)$ a uniformly bounded distance to $g$ we therefore obtain a Lipschitz retraction.
\end{proof}

\paragraph{Remark.} If one were not satisfied with the statement of item~\pref{ItemCoarseLipschitzRetract}, and if $X$ and $Y$ were to satisfy higher connectivity properties, then one could pursue the issue of a continuous Lipschitz retraction up through higher dimensional skeleta.

\bigskip

Lipschitz retracts give a handy method for verifying nondistortion:

\begin{lemma}\label{LemmaCoarseRetract}
If $X$ is a geodesic metric space and $Y \subset X$ is a rectifiable subspace, and if $Y$ is a coarse Lipschitz retract of $X$, then the inclusion $Y \hookrightarrow X$ is a quasi-isometric embedding.
\end{lemma}

\begin{proof} By definition of a rectifiable subspace, $d_X(x,y) \le d_Y(x,y)$ for all $x,y \in Y$. For the other direction, if $f \from X \to Y$ is a $(K,C)$-coarse Lipschitz retract then for all $x,y \in Y$ we have $d_Y(x,y) = d_Y(f(x),f(y)) \le K d_X(x,y) + C.$
\end{proof}

\paragraph{The Dehn function of a finitely presented group.} Consider a finitely presented group $G = \<g_i \suchthat R_j\>$, $i=1,\ldots,I$, $j=1,\ldots,J$. Let $C$ be the presentation CW-complex, with one vertex, one oriented edge $e_i$ labelled by each generator $g_i$, and one 2-cell $\sigma_j$ whose boundary is attached along the closed path given by each defining relator $R_j$, inducing an isomorphism $G \approx \pi_1(C)$. Given a \emph{relator} in $G$ --- a word $w$ in the generators that represents the identity element --- the path in $C$ labelled by $w$ is null homotopic. A \emph{Dehn diagram} for a relator $w$ is a continuous function $d \from D \to C$ defined on a closed disc $D$ such that, with respect to some CW-structure on~$D$, each vertex of $D$ maps to the vertex of $C$, each edge of $D$ maps to an edge path of~$C$, the restriction $d \restrict \bdy D$ is a parameterization of the edge path~$w$, and the restriction of $d$ to each 2-cell $\sigma$ of $D$ satisfies one of two possibilities: $d(\sigma)$ is contained in the $1$-skeleton of $C$ in which case we say that \emph{$\sigma$ collapses}; or there is a relator $R_j$ such that the restriction of $d$ to the boundary of $\sigma$ traces out the path $R_j$ and the restriction of $d$ to the interior of $\sigma$ is a homeomorphism to the interior of $\sigma_j$. The number of 2-cells of the latter type is denoted $\Area(d)$. Denote by $\Area(w)$ the minimum of $\Area(d)$ over all Dehn diagrams $d$ of~$w$. The \emph{Dehn function} of the presentation is the function $f \from \N \to \N$ where $f(m)$ equals the maximum of $\Area(w)$ amongst all words of length $\le m$. If $f'(m)$ is the Dehn function of any other finite presentation of $G$ then there exist constants $A,B,C>0$ such that
$$f'(m) \le A \, f(Bm) + Cm \quad\text{and}\quad f(m) \le A \, f'(Bm)+Cm\quad\text{for any $m \in \N$}
$$
More generally, given two functions $f,g \from \N \to \N$, we say that $f$ is \emph{dominated above} by $g$ or that $g$ is \emph{dominated below} by $f$ if there exist constants $A,B,C > 0$ such that $f(m) \le A \, g(Bm) + Cm$. In particular, for a finitely presented group $G$, either all of its Dehn functions are dominated below by an exponential function, or none of them are.

\begin{proposition}\label{PropDehnBound}
Given a finitely presented group $G$ and a finitely generated subgroup $H$, if $H$ is a Lipschitz retract of $G$ then $H$ is finitely presented and the Dehn function of $G$ is dominated below by the Dehn function of $H$.
\end{proposition}

\begin{proof} The proof is based on familiar techniques. Choose finite generating sets for $G$ and $H$ so that the generators of $H$ are a subset of the generators of $G$, and hence there is an inclusion of Cayley graphs $\Gamma_H \subset \Gamma_G$ which satisfies the hypotheses of Corollary~\ref{CorollarySubgroupComplex}. Applying that corollary we obtain a $K$-Lipschitz retraction $\pi \from \Gamma_G \to \Gamma_H$. Choose a finite set of defining relators for $G$, and let $\ell$ be an upper bound for their lengths. By attaching a 2-cell to each closed edge path of $\Gamma_H$ of length $\le K\ell$, the result is a simply connected $2$-complex on which $\pi_1(H)$ acts properly and cocompactly, from which finite presentation of $H$ follows, and from which one can derive the desired inequality of Dehn functions. 

In more detail, for any word $w$ in the generators of $H$ that defines the trivial element, there is a minimal area Dehn diagram $d$ for $w$ defined using the relators of~$G$. Let $D$ be the domain of $d$ and let $\wh D$ denote the union of the $1$-skeleton of $D$ and its collapsing $2$-cells. We may regard the restriction $d \restrict \wh D$ as mapping into the Cayley graph of $G$. The composition $\pi \composed d \restrict \wh D$ maps into the Cayley graph of $H$, and the image of the boundary of any noncollapsing $2$-cell $\sigma$ of $D$ is a closed edge path in the Cayley graph of $H$ of length at most~$K\ell$, which is the boundary of one of the 2-cells attached to $\Gamma_H$. We may therefore extend $\pi \composed (d \restrict \wh D)$ across these 2-cells to obtain a Dehn diagram for $w$ in the relators of $H$, proving that $H$ is finitely presented. Furthermore, the area of this new Dehn diagram is equal to the area of $d$ itself, proving that the Dehn function of $G$ is dominated below by the Dehn function of $H$.
\end{proof}

\paragraph{Remark} The same result is true for higher dimensional isoperimetric functions of groups with higher finiteness properties, using a similar proof, where Corollary~\ref{CorollarySubgroupComplex} is extended to (noncoarse) Lipschitz retracts between higher dimensional skeleta as remarked upon following the proof of that corollary.

\subsection{Free groups.} The notation $F = \<a_1,\ldots,a_k\>$ will mean that $F$ is the free group with free basis $\{a_1,\ldots,a_k\}$. Associated to $F$ and its given free basis is a rose $R_F$, a CW 1-complex with one vertex $x$ and with $k$ oriented edges called \emph{petals} labelled $a_1,\ldots,a_k$, such that $F$ is identified with $\pi_1(R_F,x)$ and each free basis element $a_i$ is identified with the path homotopy class of the petal with the label $a_i$. The number~$k$ is the \emph{rank} of $F$.

For the notation of this paper, $F_n = \<a_1,\ldots,a_n\>$ and $R_n=R_{F_n}$ will denote a fixed rank~$n$ free group with free basis, and its associated rose. However, for the rest of Section~\ref{SectionPrelim}, as well as in various places around the paper, we will couch some of our definitions and statements in terms of a general finite rank free group denoted $F$, in order that these definitions can apply as well to subgroups of $F_n$.

The action of the automorphism group $\Aut(F)$ on elements $g \in F$ and subgroups of $\Aut(F)$ induces an action of the outer automorphism group $\Out(F) = \Aut(F) / \Inn(F)$ on conjugacy classes of elements and subgroups. We generally use square brackets $[\cdot]$ to denote conjugacy classes, or more generally any variety of equivalence relation on which $\Out(F)$ acts.

A \emph{free factorization} of $F$, written $F = A_1 * \cdots * A_K$, is a set of subgroups $\{A_1,\ldots,A_K\}$ with pairwise trivial intersections having the property that each nontrivial $g \in F$ can be written uniquely as a product of nontrivial elements $g = a_1 \cdots a_m$ such that each $a_i$ is an element of one of $A_1,\ldots,A_K$. A~\emph{free factor system} of $F$ is a nonempty set of nontrivial subgroup conjugacy classes of the form $\A = \{[A_1],\ldots,[A_K]\}$ such that there exists a free factorization of the form $F = A_1 * \cdots * A_K * B$, where $B$ may be trivial; $\A$ is \emph{improper} if $K=1$ and $A_1=F$, otherwise $\A$ is \emph{proper}. The individual conjugacy classes $[A_1],\ldots,[A_K] \in \F$ are called the \emph{components} of $\F$, and $\F$ is \emph{connected} if it has one component.  We use literal set theory notation $\union\A$ to denote the union of the elements of $\A$, that is, for any subgroup $A \subgroup F$, we have $A \in \union\A$ if and only if $[A] \in \A$.

We define a partial order on free factor systems $\A \sqsubset \A'$ if and only if $\union\A \subset \union \A'$. Recalling the coindex defined in the introduction, we have the following elementary fact:

\begin{lemma}\label{LemmaCoindex}
For any free factor systems $\A,\A'$ in $F$, if $\A \sqsubset \A'$ then $\coindex(\A') \le \coindex(\A)$, with equality if and only if $\A=\A'$.
\qed\end{lemma}

Given an isomorphism $\rho \from F \to F'$ of finite rank free groups, there is an induced map denoted $\rho_*$, well defined up to pre and postcomposition by inner automorphisms, from conjugacy classes in $F$ to conjugacy classes in $F'$, and in particular from free factor systems in $F$ to free factor systems in $F'$. We also use the notation $\rho_*$ for similar purpose when $\rho$ is a homotopy equivalence between finite graphs.

\bigskip

The action of $\Out(F)$ on conjugacy classes of subgroups induces an action on free factor systems. The \emph{stabilizer} of a free factor system $\A$ is the subgroup
$$\Stab(\A) = \{\phi \in \Out(F) \suchthat \phi(\A)=\A\}
$$
When $\A = \{[A]\}$ is connected we write $\Stab[A] = \Stab(\A)$.

\subsection{Graphs}
\label{SectionGraphs}

In this paper a \emph{simplicial complex} $X$ will be the geodesic metric space associated with an abstract simplicial complex (see for example \cite{Spanier} Chapter 3 Section 1), in which every simplex comes equipped with coordinates making it isometric to a Euclidean simplex of side length~$1$; these metrics on simplices extend uniquely to a geodesic metric on $X$ which is proper if and only if $X$ is locally finite. The underlying $CW$-complex structure also induces a \emph{weak topology} on $X$ which equals the metric topology if and only if the complex is locally finite.

A \emph{graph} is a 1-dimensional CW-complex. A \emph{simplicial tree} is a simplicial complex which is a contractible graph. A \emph{simplicial forest} is a graph whose components are simplicial trees. On a simplicial tree we may impose a geodesic metric making each $1$-cell isometric to a closed interval in $\reals$ of any positive length, resulting in a \emph{simplicial $\reals$-tree}; an \emph{$\reals$-graph} is similarly defined. 

A map between graphs that takes each 0-cell to a 0-cell and each 1-cell either to a 0-cell or homeomorphically to a 1-cell is called a \emph{cellular map} or a \emph{simplicial map} depending on the context. A surjective cellular map is called a \emph{forest collapse} if the inverse image of any 0-cell is a tree. A bijective cellular map is called an \emph{isomorphism}, and a self-isomorphism of a graph is called an \emph{automorphism}.

Any graph or tree not homeomorphic to a circle and having no isolated ends has a unique \emph{natural} cell structure with no 0-cells of valence~$2$; phrases such as ``natural vertex'' or ``natural edge'' will refer to this natural structure. Every CW decomposition is a refinement of this natural cell structure. Every homeomorphism of the weak topology between graphs equipped with natural cell structures is an isomorphism of those structures.

\subsection{Free splittings}
\label{SectionSplittingsAndGraphs}

A \emph{splitting} of a group $G$ is a simplicial action of $G$ on a nontrivial simplicial tree $T$ such that the action is minimal, meaning no proper subtree of $T$ is invariant under the action, and for each edge $e \subset T$ and $g \in G$, if $g(e)=e$ then $g$ restricts to the identity on $e$. Formally a splitting is given by a homomorphism $\chi \from G \to \Aut(T)$ with values in the group of simplicial automorphisms of~$T$. We will denote this by $\chi \from G \act T$, or suppressing the action just $G \act T$ or even just by writing the tree $T$. We often use notation like $(g,x) \mapsto g \cdot x \in T$ to refer to the image of the action of $g \in G$ on $x \in T$; formally $g \cdot x = \chi(g)(x)$.

In this paper, the groups which act are all free of finite rank, and all splittings will be ``free splittings'' meaning that all edges have trivial edge stabilizers, and we shall immediately specialize our definitions to that context. Also, although we often work in the context of fixed rank~$n$ free group $F_n$ and its finite rank subgroups, in this section we couch all definitions in terms of an arbitrary finite rank free group, in order to accomodate free splittings of subgroups of $F_n$.

A \emph{free splitting} of a finite rank free group $F$ is a splitting $F \act T$ such that the stabilizer subgroup of every edge $e \subset T$ is trivial, meaning that if $g \in F$ and $g(e)=e$ then $g$ is the identity element of $F$.  A free splitting $F \act T$ is \emph{proper} if $\F(T)=\emptyset$, i.e.\ the stabilizer subgroup of every vertex is trivial; equivalently, the action is free and properly discontinuous. 

A \emph{conjugacy} from a free splitting $\chi \from F \act T$ to a free splitting $\chi' \from F \act T'$ is a homeomorphism $f \from T \to T'$ such that $f$ is \emph{equivariant}, meaning that $f \composed \chi(g) = \chi'(g) \composed f$ for all $g \in F$. We say that $T$ and $T'$ are \emph{conjugate} if there exists a conjugacy $f \from T \to T'$. Conjugacy is an equivalence relation amongst free splittings, and we use the notation $[T]$ to refer to the equivalence class. Note that any conjugacy $f$ is an isomorphism of natural simplicial structures, although we do not require $f$ to be an isomorphism of whatever subdivided simplicial structures $T$ and $T'$ may have. 

The outer automorphism group $\Out(F)$ has natural right actions on the set of conjugacy classes of free splittings: the action of $\phi \in \Out(F)$ on the class of a free splitting $\chi \from F \act T$ is the class of the free splitting $\chi \cdot \phi \from F \act T$ which is well-defined up to conjugacy by the formula $\chi \cdot \phi(g) = \chi(\Phi(g))$ for any choice of an automorphism $\Phi \in \Aut(F)$ representing $\phi$. We also denote this action by $[T] \cdot \phi$ or just $T \cdot \phi$. Notice that, suppressing the role of $\chi$, the image of the action $(g,x) \mapsto g \cdot x$ can be denoted as the action $(g,x) \mapsto \Phi(g) \cdot x$. The \emph{stabilizer} of the class of a free splitting $\chi_T \from F \act T$ is the group 
$$\Stab[T] = \{\phi \in \Out(F) \suchthat T \cdot \phi = T\}
$$ 

A \emph{semiconjugacy} from a free splitting $\chi \from F \act T$ to a free splitting $\chi' \from F \act T'$ is a continuous equivariant map $f \from T \to T'$ that takes vertices to vertices and restricts on each edge of $T$ to an embedding or a constant map. Since $\chi$ and $\chi'$ are minimal, any semiconjugacy between them is surjective.

\subsection{Outer space et les autres espaces} 
\label{SectionSpine}

In this section, given a finite rank free group $F$, we review the spine of outer space denoted~$\CVK_F$ on which $\Out(F)$ acts, and the spine of autre espace denoted $\Autre_F$ on which $\Aut(F)$ acts. Vertices of $\CVK_F$ are marked graphs, and vertices of $\Autre_F$ are marked graphs with a base point, modulo an appropriate equivalence in each case. For details and further references for this material see \cite{Vogtmann:OuterSpaceSurvey}. We also consider a variation on $\Autre_F$, whose vertices are marked graphs with multiple base points.

\paragraph{Marked graphs.} Choose a basis for $F$ with associated rose $R_F$, inducing an identification $\pi_1(F)=\pi_1(R_F)$. An \emph{$F$-marked graph} is a pair $(G,\rho)$ where $G$ is a finite \emph{core graph} meaning that no vertex has valence~$1$, equipped with a homotopy equivalence $\rho \from R_F \to G$ called the \emph{marking} (the usual definition requires no vertices of valence~$\le 2$, but we will allow such vertices and simply use the natural cell structure as needed). Two marked graphs $(G,\rho)$ and $(G',\rho')$ are \emph{equivalent} if there exists a homeomorphism $h \from G \to G'$ such that $\rho \composed h$ and $\rho \from G \to G'$ are homotopic. We often suppress the marking $\rho$ from the notation, using phrases like ``a marked graph $G$'', just as we often suppress the action in the notation for a free splitting; for instance, the equivalence class of $(G,\rho)$ is formally denoted $[G,\rho]$ or just $[G]$ when $\rho$ is understood. The right action of the group $\Out(F)$ on equivalence classes of marked graphs is defined by $[G,\rho] \cdot \phi = [G,\rho\composed\Phi]$ where $\Phi \from R_F \to R_F$ is a homotopy equivalence inducing $\phi \in \Out(F)=\Out(\pi_1 R_F)$. 

The universal covering construction defines a bijection between the set of conjugacy classes of proper free splittings of $F$ and the equivalence classes of $F$-marked graphs, which associates to each $F$-marked graph $(G,\rho)$ a proper, minimal action $F \act \wt G$ on the universal covering space which is well-defined up to precomposition by an inner automorphism of $F$; the inverse of this bijection associates to each proper free splitting $F \act T$ the quotient core graph $G = T/F$ with a marking $\rho \from R_F \to G$ that is well-defined up to homotopy. The action of $\Out(F)$ on conjugacy classes of free splittings descends to an action on equivalence classes of marked graphs as follows: 

Given two $F$-marked graphs $(G,\rho)$, $(G',\rho')$ and a homotopy equivalence $h \from G \to G'$, we say that $h$ \emph{preserves marking} if $h\composed\rho$ is homotopic to $\rho'$. 

Given an $F$-marked graph $(G,\rho)$, the marking $\rho \from R_F \to G$ induces a well-defined bijection $\rho_*$ between conjugacy classes in $F = \pi_1(R_F)$ and conjugacy classes in $\pi_1(G)$, and between free factor systems in $F$ and free factor systems in $\pi_1(G)$. Every proper, connected, noncontractible subgraph $H \subset G$ determines a well-defined connected free factor system in $\pi_1(G)$ denoted $[\pi_1 H]$, and one in $F_n$ denoted $[H]$ which is characterized by the equation $\rho_*[H]=[\pi_1 H]$. More generally, for every proper subgraph $H \subset G$ with noncontractible components $H_1,\ldots,H_k$ there is an associated free factor system in $F_n$ denoted $[H]=\{[H_1],\ldots,[H_k]\}$.

We record the following elementary lemma for easy reference. Given an $F$-marked graph $(G,\rho)$ and a natural subforest $E \subset G$, there is a quotient map $\pi \from G \to G/E$ that collapses each component to a point. Each component of $E$ clearly has $\ge 3$ incident edges which are not in $E$, implying that every vertex of $G/E$ has valence $\ge 3$, and so $G_i$ is indeed a core graph. The map $\pi$ is a homotopy equivalence, and so by precomposing $\pi$ with the marking on $G$ we may regard $G/E$ as a marked graph. We refer to $\pi \from G \to G/E$ as a \emph{forest collapse of marked graphs}.

\begin{lemma}\label{LemmaCollapse} Given a forest collapse $\pi \from G \to G/E=G'$ of marked graphs, let $\wt E \subset \wt G$ be the full pre-image of $E$ in the universal cover of $G$ and let $\ti\pi  \from \wt G \to \wt G'$ be the $F$-equivariant lift of $\pi$ that collapses each component of $\wt E$ to a point. 
\begin{enumerate}
\item \label{ItemEdgeBijectionCollapse} $\ti\pi$ induces a bijection between the set of edges of $\wt G - \wt E$ and the set of edges of $\wt G'$ and similarly for $\pi_{F}$.  
\item  \label{ItemEdgePath}  For any edge path  in $\wt G$ --- a finite arc, a ray, or a line --- its $\ti\pi$-image in $\wt G'$ is the edge path obtained by erasing the edges in $\wt E$ and indentifying the remaining edges via the bijection of item \pref{ItemEdgeBijectionCollapse} and similarly for edge paths in $G$.
\item \label{ItemSubTree} For each finitely generated subgroup $B \subgroup G$, the $\ti\pi$ image of the $B$-minimal subtree of $\wt G$ is the $B$-minimal subtree of $\wt G'$.
\end{enumerate}
\end{lemma}

\begin{proof} Items \pref{ItemEdgeBijectionCollapse} and \pref{ItemEdgePath} are obvious.   Item \pref{ItemSubTree} follows from items \pref{ItemEdgeBijectionCollapse} and \pref{ItemEdgePath} and the fact that the minimal $B$-subtree is characterized as the union of all axes in $T$ for the action of all nontrivial elements of $B$.
\end{proof}

\paragraph{The spine of outer space $\CVK_F$.} This is an ordered simplicial complex whose set of $0$-simplices corresponds bijectively to $F$-marked graphs up to equivalence, which using universal covering maps corresponds bijectively to proper, free splittings of $F$ up to conjugacy. For each $0$-simplex $\CVK_F$ represented by a marked graph $G$, for each $k \ge 1$, and for each properly increasing sequence of natural subforests $\emptyset \ne E_1 \subset \cdots \subset E_k$ of~$G$, there is an ordered $k$-simplex denoted $\Sigma(G;E_1 \subset \cdots \subset E_k)$ whose $i^\text{th}$ vertex is the marked graph $G_i = G_0 / E_i$ obtained from $G=G_0$ by a collapsing the forest $E_i$. In particular, for any forest collapse $G=G_0 \mapsto G_1 = G/E$ we have an oriented 1-simplex $\Sigma(G;E)$ with initial vertex $G_0$ and terminal vertex $G_1$.

\vbox{
\begin{lemma}[Facts about the spine] \quad\hfill
\label{LemmaSpineFacts}
\begin{enumerate}
\item \label{ItemSpineContractible} \cite{CullerVogtmann:moduli}
$\CVK_F$ is contractible. In particular, its 1-skeleton is connected.
\item \label{ItemSpineOrbits} \cite{CullerVogtmann:moduli}
The action of $\Out(F)$ on equivalence classes of marked graphs induces a properly discontinuous, cocompact, simplicial action $\Out(F) \act \CVK_F$.
\item \label{ItemSpineQI}
For any vertex $G$ of $\CVK_F$ the orbit map $\Out(F) \mapsto \CVK_F$ taking $G$ to $G\cdot\phi$ is a quasi-isometry from the word metric on $\Out(F)$ to the simplicial metric on the $1$-skeleton of $\CVK_F$.\qed
\end{enumerate}
\end{lemma}
}

\subparagraph{Remarks on the proof.} Item~\pref{ItemSpineQI} follows from item~\pref{ItemSpineContractible}, item~\pref{ItemSpineOrbits}, local finiteness, and the \MilnorSvarc\ lemma. 

For later purposes we remind the reader of the proofs of the various finiteness properties of the action $\Out(F) \act \CVK_F$. Cocompactness means that there are only finitely many orbits of simplices, for which it suffices to show that there are only finitely many vertex orbits. Indeed there is one orbit for each simplicial isomorphism class of core graphs whose rank equals $\rank(F)$ equipped with the natural simplicial structure, and there are only finitely many such classes because the number of natural edges is bounded above by $3 \rank(F)-3$. The subgroup of $\Out(F)$ that stabilizes each simplex is finite; for $0$-simplices this follows because the stabilizer in $\Out(F)$ equals the group of simplicial isomorphisms of the underlying core graph. Local finiteness follows because each marked graph $G$ has only finitely many natural subgraphs. From all of this it follows that the action of $\Out(F)$ on $\CVK_F$ is properly discontinuous.

\paragraph{The spine of autre espace $\Autre_F$.} An \emph{$F$-marked pointed graph} consists of a triple $(G,p,\rho)$ where $G$ is a finite core graph, $p \in G$, and $\rho \from (R_F,p_F) \to (G,p)$ is a pointed homotopy equivalence; when $\rank(F) \ge 2$ the point $p_F \in R_F$ is the unique natural vertex, whereas when $\rank(F) =1$ the point $p_F \in R_F$ is chosen arbitrarily. We adopt the convention that $G$ is equipped with the \emph{relatively 
natural cell structure}, obtained from the natural cell structure by subdividing at the point $p$. Throughout Sections 1.6 and 2 we refer to this convention with phrases like \lq relatively natural vertex\rq,\lq relative natural subforest\rq, etc. Two $F$-marked pointed graphs $(G,p,\rho)$, $(G',p',\rho')$ are \emph{pointed equivalent} if there exists a homeomorphism $h \from (G,p) \to (G',p')$ such that $h \rho$ and $\rho' \from (R_F,p_F) \to (G',p')$ are homotopic relative to $p_F$. 
One can also consider \emph{proper, pointed free splittings} of $F$, that is, free splittings $F \act T$ equipped with a base point $P \in T$, modulo the equivalence relation of \emph{base point preserving conjugacy}. 

The spine of autre espace is an ordered simplicial complex $\Autre_F$ whose $0$-simplices correspond bijectively to $F$-marked pointed graphs up to pointed equivalence, which by universal covering space theory correspond bijectively to proper, pointed free splittings of $F$ up to base point preserving conjugacy. For each $0$-simplex represented by a pointed marked graph $(G,p,\rho)$, for each $k \ge 1$, and for each properly increasing sequence of \emph{relatively} natural subforests $\emptyset \ne E_1 \subset \cdots \subset E_k$ of $G$, there is an ordered $k$-simplex whose $i^{\text{th}}$ vertex is represented by the core graph $G/E_i$, with base point and marking obtained by pushing forward $p$ and $\rho$ via the forest collapse $f_i \from G \to G/E_i$, yielding $p_i = f_i(p) \in G/E_i$ and $\rho_i = f_i \composed \rho \from (R_F,p_F) \to (G/E_i,p_i)$. In particular, a 1-simplex of the ordered simplicial complex $\Autre_F$ with initial $0$-simplex represented by $(G,p,\rho)$ and terminal $0$-simplex represented by $(G',p',\rho')$ is represented by a forest collapse $h \from G \to G'$ taking $p$ to $p'$ such that $h \rho$ and $\rho' \from (R_F,p_F) \to (G',p')$ are homotopic rel $p_F$.

The action of the group $\Aut(F)$ on $\Autre_F$ is induced by the action on the vertex set which is defined by precomposing a marking $\rho \from (R_F,p_F) \to (G,p)$ with a homotopy equivalence $(R_F,p_F) \mapsto (R_F,p_F)$ realizing a given automorphism. The complex $\Autre_F$ is connected, and the action of $\Aut(F)$ is properly discontinuous and cocompact, with one orbit for each relatively natural isomorphism class of pairs $(G,p)$ where $G$ is a core graph of rank equal to $\rank(F)$ and $p \in G$. Applying the \MilnorSvarc\ lemma, the orbit map $\Aut(F_n) \mapsto \Autre_n$ is a quasi-isometry.

For later purposes we give a proof of connectivity of $\Autre_F$ by reducing it to connectivity of $\CVK_F$. There is a simplicial map $\Autre_F \mapsto \CVK_F$ which ``forgets the base point'', defined on the $0$-skeleton by taking the equivalence class of the pointed marked graph $(G,p,\rho)$ to the equivalence class of the unpointed marked graph $(G,\rho)$. The preimage in $\Autre_F$ of the vertex represented by $(G,\rho)$ is a connected 1-complex identified with the first barycentric subdivision of the natural cell structure on the universal covering space $\wt G$: choosing the basepoint of $\wt G$ to be any natural vertex determines a $0$-simplex of the preimage; choosing the basepoint to lie in the interior of any natural edge of $\wt G$ also determines a $0$-simplex; and allowing the basepoint to move from the interior of a natural edge to either of its natural endpoints determines a $1$-simplex. To prove that $\Autre_F$ is connected, it therefore suffices to show that each $1$-simplex of $\CVK_F$ can be lifted via the forgetful map to a $1$-simplex of $\Autre_F$. Consider a 1-simplex in $\CVK_F$ represented by a forest collapse of marked graphs $h \from G \mapsto G'$, with markings $\rho \from R_F \to G$ and $\rho' \from R_F \to G'$. By homotoping $\rho$ and $\rho'$ we may assume that $p =\rho(p_F) \in G$ is a natural vertex and that $\rho' = h \composed \rho$ and so $p' = h(p) \in G'$ is a natural vertex; and with these assumptions it follows that the pointed marked graphs $(G,p,\rho)$, $(G',p',\rho')$ represent endpoints of a 1-simplex in $\Autre_F$ which is a lift of the given 1-simplex of~$\CVK_F$.

\paragraph{The spine of $M$-pointed autre espace $\Autre[M]_F$.} Fix an integer $M \ge 1$. For use in Section~\ref{SectionFreeSplittingSubcomplex} we need a variation of the autre espace $\Autre_F$ with its action by $\Aut(F)$, which allows for an $M$-tuple of base points in a marked graph. The group that will act, denoted $\Aut^M(F)$, is the subgroup of the $M$-fold direct sum $\Aut(F) \oplus \cdots \oplus \Aut(F)$ consisting of all $M$-tuples $(\Phi_1,\ldots,\Phi_M)$ whose images in $\Out(F)$ are all equal. 

Define an \emph{$M$-pointed $F$-marked graph} to be a tuple 
$$(G;p_1,\ldots,p_M;\rho_1,\ldots,\rho_M) \quad\text{or, in shorthand,}\quad (G;(p_m);(\rho_m)), \quad m=1,\ldots,M
$$
where $G$ is a finite core graph, $p_m \in G$ and $\rho_m \from (R_F,p_F) \to (G,p_m)$ is a homotopy equivalence of pairs for each $m$, and the homotopy equivalences $\rho_1,\ldots,\rho_M \from R_F \to G$ are all in the same free homotopy class. The \emph{relatively natural} cell structure on $G$ is the one whose vertex set is the union of the natural vertices of $G$ and the set $\{p_1,\ldots,p_M\}$. Given another $M$-pointed $F$-marked graph $(G';(p'_m);(\rho'_m))$, we say that the two are \emph{pointed equivalent} if there is a homeomorphism $h \from (G;p_1,\ldots,p_M) \to (G';p'_1,\ldots,p'_M)$ such that for each $m$ the homotopy equivalences $h \composed \rho_m$, $\rho'_m \from (R_F,p_F) \to (G',p'_m)$ are homotopic relative to $p_F$.

Alternatively, define a \emph{proper, $M$-pointed free splitting} of $F$ to be a free splitting $F \act T$ equipped with an $M$-tuple of points $P_1,\ldots,P_M \in T$, modulo the equivalence relation of \emph{$M$-tuple preserving conjugacy}. Again the universal covering map produces a bijection between equivalence classes of $M$-pointed $F$-marked graphs and $M$-tuple preserving conjugacy classes of proper, $M$-pointed free splittings. We shall describe this bijection explicitly.  Consider an $M$-pointed $F$-marked graph $(G;p_m;\rho_m)$. In the universal cover $\wt G=T$ choose a lift $P_1$ of $p_1$, which by universal covering space theory determines an action $F \act T$ which is a proper free splitting. Also, let $\wt R_F$ be the universal covering space of $R_F$, choose a lift $P_F$ of $p_F$, determining an action $F \act \wt R_F$. Let $\ti \rho_1 \from (\wt R_F,P_F) \to (T,P_1)$ be the unique $F$-equivariant lift of $\rho_1 \from (R_F,p_F) \to (G,p_1)$. For each $m=2,\ldots,M$, any free homotopy between $\rho_1$ and $\rho_m \from R_F \to G$ lifts to a free homotopy between $\ti\rho_1$ and a certain lift $\ti\rho_m \from \wt R_F \to T$ of $\rho_m$. Let $P_m = \ti \rho_m(P_F)$. The $M$-tuple $(P_1,\ldots,P_M)$ in $T$ together with the action $F \act T$ is a proper $M$-pointed free splitting, and this construction defines the desired bijection.

Define an ordered simplicial complex $\Autre[M]_F$ whose $0$-simplices are $M$-pointed $F$-marked graphs up to pointed equivalence, or alternatively $M$-pointed free splittings up to $M$-tuple preserving conjugacy. An ordered $k$-simplex of $\Autre[M]_F$ is defined very much as for $\Autre_F$, starting with a choice of a $0$-simplex represented by $(G;p_m;\rho_m)$ and a choice of a strictly increasing sequence of relatively natural subforests $\emptyset \ne E_1 \subset \cdots \subset E_k \subset G$; the $i^{\text{th}}$ vertex has underlying graph $G/E_i$, with points and markings obtained by pushing forward $p_1,\ldots,p_M$ and $\rho_1,\ldots,\rho_M$. In particular there is an edge from the vertex represented by $(G;p_m;\rho_m)$ to the vertex represented by $(G';p'_m;\rho'_m)$ if and only if there is a relatively natural forest collapse $h \from G \mapsto G'$ such that for each $i=1,\ldots,m$ we have $h(p_m)=p'_m$ and the homotopy equivalences $\rho'_m$, $h \composed \rho_m \from (R_F,p_F) \to (G',p'_m)$ are homotopic rel $p_F$. 

Connectivity of $\Autre[M]_F$ is proved by induction on $M$, as follows. The map which forgets the last point $p_m$ and marking $\rho_M$ is a simplicial map $\Autre[M]_F \mapsto \Autre[M-1]_F$. As in the proof of connectivity of $\Autre_F$, the pre-image in $\Autre[M]_F$ of the vertex of $\Autre[M-1]_F$ represented by $(G;p_1,\ldots,p_{M-1};\rho_1,\ldots,\rho_{M-1})$ is simplicially isomorphic to the universal covering space of $G$ equipped with the lift of the first barycentric subdivision of the relatively natural cell structure of $(G,p_1,\ldots,p_{M-1})$. Each 1-simplex of $\Autre[M-1]_F$ lifts via the forgetful map to a 1-simplex of $\Autre[M]_F$. The induction step follows immediately.

There is a natural group action $\Aut^M(F) \act \Autre[M]_F$, where the action of $(\Phi_1,\ldots,\Phi_M)$ on a $0$-simplex represented by $(G;(p_m);(\rho_m))$ is the $0$-simplex represented by $(G;(p_m);(\rho'_m))$ where the marking $\rho'_m \from (R_F,p_F) \mapsto (G,p_m)$ is obtained by precomposing the marking $\rho_m \from (R_F,p_F) \mapsto (G,p_m)$ with a homotopy equivalence of the pair $(R_F,p_F)$ that represents the automorphism $\Phi_m \in \Aut(F)$. Since all of the automorphisms $\Phi_m$ are in the same outer automorphism classes,
and since all of the marking $\rho'_m$ are freely homotopic, it follows that all of the composed markings $\rho'_m$ are still freely homotopic, and so $(G,(p_m);(\rho'_m))$ does indeed represent a $0$-simplex of $\Autre[M]_F$. The proof that this action is properly discontinuous and cocompact with finite cell stabilizers is a simple generalization of the proof for the action $\Aut(F) \act \A_F$.

\section{The $\Aut(F_{n-1})$ subgroup of $\Aut(F_n)$} 
\label{SectionAutNondistortion}

In this section we prove Theorem~\ref{TheoremAutRetract}: the naturally embedded subgroup $\Aut(F_{n-1}) \subgroup \Aut(F_n)$ is a Lipschitz retract of $\Aut(F_n)$. Fixing the free basis $F_n = \<a_1,\ldots,a_{n-1},a_n\>$, and identifying $F_{n-1} = \<a_1,\ldots,a_{n-1}\> \subgroup F_n$, this subgroup consists of all $\phi \in \Aut(F_n)$ that preserve the subgroup $\<a_1,\ldots,a_{n-1}\>$ and fix the element $a_n$.    We assume that the edges of the rank $n$ rose $R_n$ correspond to $a_1,\ldots,a_n$ and that the rank $(n-1)$ rose $R_{n-1}$ is identified with the subrose of $R_n$ corresponding to $a_1,\ldots,a_{n-1}$.     

\vspace{.1in} \noindent {\bf The $\Aut(F_{n-1})$ equivariant embedding $ \Autre_{n-1}   \overset{j} \inject \Autre_n$.} The embedding $j$ is defined as follows.  Start 
with an $F_{n-1}$-marked pointed graph $(H,p,\rho_H)$ representing a point in 
the spine $ \Autre_{n-1}$, then attach to $p$ a loop of length $1$ to form a marked 
graph $(G,p,\rho_G)$, where $\rho_G \restrict R_{n-1}$ equals $\rho_H$, and $\rho_G$ takes the 
$a_n$ loop of $R_n$ around the newly attached loop of G.

The image  $j(\Autre_{n-1}) \subset   \Autre_n$ is invariant under the action of the subgroup $\Aut(F_{n-1}) \subgroup \Aut(F_n)$. Since $\Aut(F_n)$ and $\Aut(F_{n-1})$ act properly discontinuously and cocompactly on $\Autre_n$ and $j(\Autre_{n-1})$ respectively, the hypotheses of Corollary~\ref{CorollarySubgroupComplex}~\pref{ItemZeroLipschitzRetract} apply. We conclude that in order to prove that $\Aut(F_{n-1})$ is a Lipschitz retract of $\Aut(F_n)$ it suffices to prove that the $0$-skeleton of $j(\Autre_{n-1})$ is a Lipschitz retract of the $0$-skeleton of~$\Autre_n$.

\paragraph{Retracting the $0$-skeleton of $\Autre_n$ to the $0$-skeleton of  $j(\Autre_{n-1}$). }     Given an $F_n$-marked pointed graph $(G,p,\rho)$ representing an arbitrary $0$-simplex of  $\Autre_n$ we will construct an   $F_{n-1}$-marked pointed graph $r(G,p,\rho)$ representing a $0$-simplex of  $\Autre_{n-1}$ such that the restriction of $rj$ to the  $0$-skeleton of $\Autre_{n-1}$ is the identity.   The map $R =jr $   is then the desired retraction.

Let $(\ti G, \ti p)$ be the universal cover with base point $\ti p$ that projects to $p$ and with the action  of $F_n$ determined by $\rho$ and the choice of $\ti p$.  Let $S \subset \ti G$ be the minimal $F_{n-1}$-subtree  and let $\ti q\in S$ be the nearest point  to $\ti p$ in $S$.   The quotient space $(K, q)$  of $(S, \ti q)$  by the action of $F_{n-1}$ is an $F_{n-1}$-marked pointed graph. The marking $\rho_K$ is an immersion that  maps the edge in $R_{n-1}$ corresponding to $a_j$ to the projected image  of the path in $S$ from $\ti q$ to $\ti q \cdot a_j$.   Define $r(G,p,\rho) = (K,q,\rho_K)$.

    Given a $0$-simplex $w$ of $\Autre_{n-1}$, letting 
$j(w)=(G,p,\rho_G)$ as in the definition of $j$, there is a connected core 
subgraph $H \subset G$ that    contains $p$  such that $\rho \restrict R_{n-1} :R_{n-1} \to H$ is a homotopy equivalence and such  $(H,p,\rho\restrict  R_{n-1})$ represents $w$.  In this case, the minimal $F_{n-1}$-subtree $S \subset \ti G$   is a component of the full pre-image of $H$ and $\ti p = \ti q \in S$.   Thus $r(G,p,\rho) = (H,p,\rho\restrict  R_{n-1})$ and $rj(w) = w$.

 \paragraph{The Lipschitz constant of $R$  is $1$.}   If $w$ and $w'$ are $0$-simplices of $\Autre_{n-1}$ that bound an edge in $\Autre_n$ then $j(w)$ and $j(w')$   bound an edge in   $\Autre_{n}$.  Thus $j$ has Lipschitz constant $1$.  We will prove that $r$ has Lipschitz constant $1$ and hence that $R=rj$ has Lipschitz constant $1$. It suffices to show that  if $(G,p,\rho)$ and $(G',p',\rho')$ are the endpoints of an edge in $\Autre_n$ then $r(G,p,\rho)= (K,q,\rho_K)$ and $r(G',p',\rho') =  (K',q',\rho_{K'})$ are either equal or are the endpoints of an edge in $\Autre_{n-1}$.  In other words, if $(G',p',\rho')$ is obtained from  $(G,p,\rho)$ by collapsing each component of a non-trivial relatively natural forest $F \subset G$ to a point then $(K',q',\rho_{K'})$ is obtained from  $(K,q,\rho_{K})$ by collapsing each component of a (possibly trivial) relatively natural forest  $F'$ to a point.  
 
    Let $\ti F \subset \ti G$ be the full pre-image of $F$ in the universal cover $\ti G$ of $G$ and, as above, let $S$ be the minimal $F_{n-1}$ subtree of $\ti G$.    Lift the quotient map $\pi_F : G\to G'$ to the pointed universal covers, obtaining a map $ \pi_{\ti F} : (\ti G,\ti p) \to (\ti G',\ti p')$  that collapses each component of $\ti F$ to a point.    Lemma~\ref{LemmaCollapse} implies that  $S' =  \pi_{\ti F}(S)$ is the minimal $F_{n-1}$ subtree of $\ti G'$ and that  the    $\pi_{\ti F}$-image of the arc connecting $\ti p$ to $\ti q$ is an arc in $\ti G'$ connecting $\ti p'$ to $\ti q' := \pi_{\ti F}(\ti q)$ that intersects $S'$ only in $\ti q'$, proving that  $\ti q'  \in S'$ is the nearest point to $\ti p'$ in $S'$.      Moreover, the arc in $S'$ from $\ti q'$ to $\ti q' \cdot a_j$  is the  $\pi_{\ti F}$-image   of the arc in $S$ from $\ti q$ to $\ti q \cdot a_j$.       Let $\ti F'' = \ti F \cap S$ and let $F'' \subset K$ be the projected image of $\ti F''$.    Then $(K',q',\rho_{K'})$  is obtained from  $(K,q,\rho_{K})$ by collapsing each component of $F''$ to a point.  However, the components of  $F'$ need not be unions of relatively natural edges of $K$.    We therefore define the relatively natural forest $F'$ by including a relatively natural edge in $F'$ if and only if it is entirely contained in $F''$.   It is easy to check that replacing $F''$ by $F'$ does not change the $0$-simplex of  $\Autre_{n-1}$ obtained by collapsing components to points.   If $F'$ is empty then $r(G,p,\rho)= r(G',p',\rho')$.   Otherwise,  collapsing the components of $F'$ to  points defines an edge in $\Autre_{n-1}$ with endpoints $r(G,p,\rho)$ and $r(G',p',\rho')$.

\section{Stabilizers of free splittings in $\Out(F_n)$}
\label{SectionTreeStabilizers}

In this section we prove Theorem~\ref{TheoremFreeSplitting} --- for any free splitting $F_n \act T$, the subgroup $\Stab[T]$ is a Lipschitz retract of $\Out(F_n)$. 

Throughout this section we drop the subscript $n$, using $\CVK$ to denote the spine of rank~$n$ outer space. By applying Corollary~\ref{CorollarySubgroupComplex}, it is sufficient to find a coarse Lipschitz retraction from the spine of outer space $\CVK$ to a nonempty, connected subcomplex which is invariant by and cocompact under the action of $\Stab[T]$. In Section~\ref{SectionFreeSplittingSubcomplex} we construct such a subcomplex $\CVK^T \subset \CVK$; the construction and verification of the properties needed to apply Corollary~\ref{CorollarySubgroupComplex} will be obtained by showing that $\CVK^T$ is simplicially isomorphic to a certain Cartesian product of multipointed autre espaces. In Section~\ref{SectionSubcomplexRetract} we construct a coarse Lipschitz retraction $\CVK \mapsto \CVK^T$.  As in the proof of Theorem~\ref{TheoremAutRetract}, there are two steps in defining the retraction.  The first uses minimal subtrees to define marked subgraphs realizing a given free factor system.  In the context of $  \Autre_n$,  the free factor system is just $[F_{n-1}]$; in our current context it is $\F(T)$.   In the second step, the marked subgraphs are extended to marked graphs of full rank.  This is essentially trivial in the context of $\Autre_n$  because the marked subgraph is naturally pointed and the full marked graph is obtained by attaching a marked loop to the basepoint.  In our current context, this extension step is considerably more subtle; see the paragraph headed \lq Defining $R$\rq\ and the remark at the end of the section.

\subsection{Free splittings, their stabilizers, and their subcomplexes}
\label{SectionFreeSplittingSubcomplex}

Given a free splitting $F_n \act T$, we construct a subcomplex $\CVK^T \subset \CVK$ which is a geometric model for $\Stab[T]$ as follows:
\begin{lemma}\label{LemmaCVKT}
$\CVK^T$ is a connected flag subcomplex of $\CVK$ preserved by $\Stab[T]$, and the action $\Stab[T] \act \CVK^T$ is properly discontinuous and cocompact. It follows that $\Stab[T]$ is finitely generated, and that for any vertex $[S]$ of $\CVK^T$ the orbit map $\Stab[T] \mapsto \CVK^T$ taking $\phi$ to $[S]\cdot\phi$ is a quasi-isometry.
\end{lemma}
\noindent
The second sentence follows from the first by the \MilnorSvarc\ lemma. 

The definition of $\CVK^T$ and proof of Lemma~\ref{LemmaCVKT} takes up the rest of this section, during which we fix the free splitting $F \act T$ and an $F$-invariant natural simplicial structure on~$T$. For any $\phi \in \Stab[T]$ and any representative $\Phi \in \Aut(F_n)$ we denote $h_\Phi \from T \to T$ to be the simplicial homeomorphism uniquely characterized by the property that $h_\Phi \composed \Phi(g) = g \composed h_\Phi \from T \to T$ for all $g \in F_n$.

\begin{definition}[Definition of $\CVK^T$] \label{DefSigmaT} 
Given a proper free splitting $F_n \act S$, a semiconjugacy $f \from S \to T$ is  \emph{tight} if the following~hold:
\begin{enumerate}
\item \label{ItemEdgelets}
For any $V \in \V^\nt(T)$ the subgraph $S^V := f^\inv(V)$ is the $\Stab(V)$-minimal subtree of~$S$. 
\item \label{ItemCoedges}
The map $f$ is injective over the set $T - \V^\nt(T)$. 
\end{enumerate}
The $0$-simplices $\CVK^T_0$ are those proper free splittings $F_n \act S$ for which there exists a tight semiconjugacy $f \from S \to T$. Define $\CVK^T \subset \CVK$ to be the flag subcomplex with $0$-skeleton $\CVK_0^T$. 
\end{definition}

$\Stab[T]$ invariance of $\CVK^T$ follows from invariance of $\CVK_0^T$: given $\phi \in \Stab[T]$ and a $0$-simplex $[S] \in \CVK_0^T$ with tight semiconjugacy $f \from S \to T$, for any representative $\Phi \in \Aut(F)$ it follows that $h_\Phi\composed f \from S \to T$ is tight semiconjugacy for $[S] \cdot \phi$ which is therefore in $\CVK_0^T$.

After establishing some notation used in this section and the next, we state Lemma~\ref{LemmaFSDescription} which provides an explicit description of $\Stab[T]$ and $\CVK^T$ from which the remaining clauses of Lemma~\ref{LemmaCVKT} immediately follow. 

\smallskip
\textbf{Notational conventions for $T$ and $\CVK^T$.} 
Let $\V^\nt=\V^\nt(T)$ be the vertices of $T$ with nontrivial stabilizer and let $L$ be the number of orbits of the action $F_n \act \V^\nt$. Let $X = T / F_n$ be the quotient graph of groups, with simplicial structure inherited from the natural structure on $T$. Let $X^\nt = \{v_1,\ldots,v_L\} \subset X$ be image of $\V^\nt$, in other words, the vertices of $X$ that are labelled by nontrivial subgroups. For each $\ell=1,\ldots,L$, let $M_\ell$ denote the valence of $v_\ell$ in~$X^\nt$, and let $\eta_{\ell 1},\ldots,\eta_{\ell M_\ell}$ enumerate the oriented edges of $X^\nt$ with initial vertex $v_\ell$. Choose a lift $V_\ell \in \V^\nt$ of each $v_\ell$, and let $A_\ell = \Stab(V_\ell)$, so $\F(T) =  \{[A_1],\ldots,[A_L]\}$. The action of $A_\ell$ on oriented edges of $T$ with initial vertex $V_\ell$ has $M_\ell$ orbits; choose orbit representatives denoted $\ti\eta_{\ell 1},\ldots,\ti\eta_{\ell M_\ell}$, with $\ti\eta_{\ell m}$ mapping to $\eta_{\ell m}$. 

With notation fixed as in the previous paragraph, henceforth we will let $\ell$ take values in $1,\ldots,L$ and, when the value of $\ell$ is fixed, we will let $m$ take values in $1,\ldots,M_\ell$. 

Define the \emph{co-edge forest} of $T$ to be the $F_n$-forest $\wh T$ obtained from $T$ by removing the vertex subset $\V^\nt$ and taking the completion: for each $\ell,m$ and each $g \in F_n$, after removing the initial vertex $g \cdot V_\ell$ of the oriented edge $g \cdot \ti\eta_{\ell m} \subset T$, we add a valence $1$ vertex denoted $\wh V_{\ell m g} \in \wh T$ in its place. There is a natural equivariant simplicial quotient map $\wh T \to T$ that is $M_\ell$-to-one over each $g \cdot V_\ell$ and is otherwise injective, taking each $\wh V_{\ell m g}$ to $g \cdot V_\ell$.

Given a tight semiconjugacy $f \from S \to T$, denoting the map to the quotient marked graph as $q \from S \to S/F_n=G$, and recalling from Definition~\ref{DefSigmaT} the minimal subtree $S^{V_\ell}$ for the restricted action of $\Stab(V_\ell)$, the subgraphs $H_\ell = q(S^{V_\ell})$ are pairwise disjoint, connected, core subgraphs representing the free factor system $\F(T) = \{[A_1],\ldots,[A_L]\}$. Letting $[S]=[G]$ be the initial $0$-simplex of some $k$-simplex $\Sigma(G;E_1,\ldots,E_k) \subset \CVK$, we have $\Sigma(G;E_1,\ldots,E_k) \subset \CVK^T$ if and only if $E_i \subset H_1 \union\cdots\union H_L$ for each $i=1,\ldots,k$, if and only if $E_k \subset H_1 \union\cdots\union H_L$. 

\smallskip\textbf{Explicit description of $\Stab[T]$.} For any $\phi \in \Stab[T]$ and any choice of $\Phi \in \Aut(F_n)$ representing $\phi$, the simplicial homeomorphism $h_\Phi \from T \to T$ descends to a corresponding simplicial homeomorphism of the quotient graph of groups $X$ which is is well-defined independent of the choice of $\Phi$. This defines a simplicial action $\Stab[T] \act X$ whose kernel is a finite index subgroup denoted $\Stab^f[T]$.

\begin{lemma} \label{LemmaFSDescription}
For each free splitting $F \act T$ with notation as above, there exists a group isomorphism 
\begin{equation}
\Stab^f[T] \approx \Aut^{M_1}(A_1) \oplus \cdots \oplus \Aut^{M_L}(A_L)
\end{equation}
and a simplicial homeomorphism
\begin{equation}\label{EqProduct}
\I \from \CVK^T \approx \Autre[M_1]_{A_1} \cross \cdots \cross \Autre[M_L]_{A_L}
\end{equation}
such that the actions $\Stab^f[T] \act \CVK^T$ and $\Aut^{M_\ell}(A_\ell) \act\Autre[M_\ell]_{A_\ell}$ agree, for each $\ell$. 
\end{lemma}
\noindent
Lemma~\ref{LemmaCVKT} is an immediate consequence of Lemma~\ref{LemmaFSDescription} and the fact, proved in Section~\ref{SectionSpine}, that $\Aut^{M_\ell}(A_\ell) \act \Autre[M_\ell]_{A_\ell}$ is a properly discontinuous, cocompact action on a connected complex. We shall construct the simplicial homeomorphism $\I$ only on $1$-skeleta, which suffices for application to Lemma~\ref{LemmaCVKT}; the full description of $\I$ is left to the reader. 

\smallskip{\bfseries The homeomorphism $\I$ on $0$-skeleta.} Given $[S] \in \CVK_0^T$ with tight semiconjugacy $f \from S \to T$, for each $\ell,m$ let $\wh \eta_{\ell m}$ be the unique natural oriented edge of $S$ such that $f(\wh \eta_{\ell m})=\ti\eta_{\ell m}$. Let $Q_{\ell m} \in S^{V_\ell}$ be the initial vertex of $\wh \eta_{\ell m}$; note that if $g \in F_n$, the point $g \cdot Q_{\ell m} \in S^{g \cdot V_\ell}$ is the initial vertex of~$g \cdot \wh \eta_{\ell m}$. Taking the action $A_\ell \act S^{V_\ell}$ together with the $M_\ell$ tuple of points $Q_{\ell 1},\ldots,Q_{\ell M_\ell} \in S^{V_\ell}$, we obtain an $M$-pointed proper $A_\ell$ splitting representing a $0$-simplex $\I_\ell[S] \in \Autre[M_\ell]_{A_\ell}$. Define the $0$-simplex $\I[S] = (\I_1[S],\ldots,\I_L[S]) \in \Autre[M_1]_{A_1} \cross \cdots \cross \Autre[M_L]_{A_L}$. 

For the inverse of $\I$, a $0$-simplex of $\Autre[M_1]_{A_1} \cross \cdots \cross \Autre[M_L]_{A_L}$ is determined by the following data for each $\ell$: a proper, free splitting $A_\ell \act S_\ell$, and an $M_\ell$-tuple of points $Q_{\ell 1},\ldots,Q_{\ell M_\ell} \in S_\ell$. From this data we construct a $0$ simplex $[S] \in \CVK^T$ as follows. Defining $S^{V_\ell} = S_\ell$, the $L$-tuple of actions $A_\ell \act S^{V_\ell}$ extends uniquely (up to conjugacy of $F_n$ actions on graphs) to an ambient action $F_n \act \bigcup_{V \in \V^\nt} S^V$ on a forest having one component $S^V$ for each $V \in \V^\nt$ so that $S^V$ is stabilized by $\Stab[V]$ and so that the action of $\A_\ell = \Stab[V_\ell]$ on $S^{V_\ell}=S_\ell$ is the action originally given; a formal description of this action, which we leave to the reader, comes from the standard representation theory formula for inducting a subgroup representation up to a representation of the ambient group. 

Define $S$ be the quotient graph of the disjoint union of the co-edge forest $\wh T$ and the forest $\bigcup_{V \in \V^\nt} S^V$ obtained by equivariantly identifying each valence~$1$ vertex $\wh V_{\ell m g} \in \wh T$ with the point $g \cdot Q_{\ell m} \in S^{g \cdot V_\ell}$. The actions of $F_n$ on $\wh T$ and $\bigcup_{V\in\V^\nt} S^V$ clearly combine to give an action $F_n \act S$. The equivariant maps $\wh T \mapsto T$, $\bigcup_{V\in\V^\nt} S^V \mapsto \V^\nt \subset T$ clearly induce an equivariant map $f \from S \mapsto T$. The graph $S$ is a tree because $T$ is a tree and the inverse image of each point of $T$ is a subtree of $S$. Minimality of the action $F \act S$ follows from minimality of the actions $A_\ell \act S_\ell$ and $F_n \act T$. The map $S \mapsto T$ is a tight semiconjugacy by construction. 

Clearly these two constructions are inverse functions of each other, and so $\I$ is indeed a bijection of $0$-skeleta in~\pref{EqProduct}.

\smallskip{\bfseries The homeomorphism $\I$ on $1$-skeleta.} Consider a proper, free splitting $F \act S$ and quotient marked graph $q \from S \to S/F=G$ representing a $0$-simplex $[S]=[G] \in \CVK^T$. It suffices to prove that $\I$ restricts to a bijection between the set of terminal $0$-simplices of those $1$-simplices in $\CVK^T$ with initial point $[G]$, and the set of terminal $0$-simplices of those $1$-simplices in $\Autre[M_1]_{A_1} \cross\cdots\cross \Autre[M_L]_{A_L}$ with initial point $\I[G]$. The image $0$-simplex $\I[G] = (\I_1[G],\ldots,\I_L[G])$ is as described above, where each $\I_\ell[G]$ is represented by an $M_\ell$-pointed $A_\ell$-marked graph with underlying graph $H_\ell = q(S^{V_\ell})$ and $M_\ell$-tuple of points $q(Q_{\ell m}) \in H_\ell$; the $M_\ell$-tuple of pointed markings is understood. The $1$-simplices of $\CVK^T$ with initial vertex $[G]$ have the form $\Sigma(G;E)$ where $E \subset H_1 \union\cdots\union H_L$ is a nonempty natural subforest of $G$. Each intersection $E_\ell = E \intersect H_\ell$ determines an ordered simplex of $\Autre[M_\ell]_{A_\ell}$ having initial $0$-simplex $\I_\ell[S]$, of dimension $0$ or $1$ depending on whether $E_\ell$ is empty or nonempty, at least one of which has dimension~$1$ because at least one of $E_\ell$ is nonempty; taking these together we obtain a $1$-simplex $\I(\Sigma(G;E)) \in \Autre[M_1]_{A_1} \cross\cdots\cross \Autre[M_L]_{A_L}$ with initial point $\I[G]$. Clearly an arbitrary such $L$-tuple $(E_1,\ldots,E_L)$ can occur, since we can choose any subforest $E_\ell \subset H_\ell$, empty or not, as long as at least one is nonempty, and then we take $E=E_1\union\cdots\union E_L$. We therefore obtain a bijection of $1$-simplices with initial point $[G]$ and $1$-simplices with initial point $\I[G]$, and from the construction the terminal points of these $1$-simplices clearly correspond under $\I$, as required.

\smallskip{\bfseries The action isomorphism for $\Stab^f[T]$.} For each $\phi \in \Stab^f[T]$, and each $\ell$, all representatives $\Phi \in \Aut(F_n)$ of $\phi$ have the property that $h_\Phi$ preserves the $F$-orbit of each~$V_\ell$, and so amongst these representatives there exists one that satisfies $h_\Phi(V_\ell)=V_\ell$; any such $\Phi$ preserves $A_\ell = \Stab(V_\ell)$ and therefore gives an outer automorphism of $A_\ell$ which is well-defined independent of the choice of $\Phi$, thereby producing a homomorphism $\theta_\ell \from \Stab^f[T] \to \Out(A_\ell)$. For each $\ell,m$, all representatives $\Phi \in \Aut(F_n)$ of $\phi$ for which $h_\Phi(V_\ell)=V_\ell$ have the property that $h_\Phi$ preserves the orbit of the oriented edge $\ti\eta_{\ell m}$ under the action of $\Stab(V_\ell)$, and so amongst these representatives there exists one which satisfies the further restriction that $h_\Phi(\ti\eta_{\ell m}) = \ti\eta_{\ell m}$. This gives a homomorphism $\Theta^m_\ell \from \Stab^f[T] \to \Aut(A_\ell)$ well-defined independent of the restricted choice of $\Phi$. By construction each $\Theta^m_\ell$ is a lift of $\theta_\ell$, and so letting $m$ vary we obtain a homomorphism $\Theta_\ell = (\Theta^m_1,\ldots,\Theta^{M_\ell}_\ell) \from \Stab^f[T] \to \Aut^{M_\ell}(A_\ell)$. Letting $\ell$ vary we a homomorphism $\Theta = (\Theta_1,\ldots,\Theta_L) \from \Stab^f[T] \to \Aut^{M_1}(A_1) \oplus \cdots \oplus \Aut^{M_L}(A_L)$. 

By construction, the two actions of $\Stab^f[T]$ of $\Autre[M_1]_{A_1} \cross \cdots \cross \Autre[M_L]_{A_L}$ coincide --- one coming from the homomorphism $\Theta$, and the other from the action $\Stab^f[T] \act \CVK^T$ followed by the simplicial isomorphism $\CVK^T \xrightarrow{\I} \Autre[M_1]_{A_1} \cross \cdots \cross \Autre[M_L]_{A_L}$. From injectivity of the action $\Stab^f[T] \act \CVK^T$ it follows that $\Theta$ is injective. 

For surjectivity of $\Theta$ we switch to the marked graph point of view. Pick a marked graph $G$ representing a vertex of $\CVK^T$, with subgraphs $H_1,\ldots,H_L$ representing $[A_1],\ldots,[A_L]$. By collapsing a maximal forest of each $H_\ell$ to a point we may assume that $H_\ell$ is a rose whose base vertex $q_\ell$ equals its frontier in $G$. By projecting a tight semiconjugacy $\wt G \mapsto T$ we obtain a quotient map $G \mapsto X$ which collapses each $H_\ell$ to $v_\ell$ and is otherwise injective. For $\ell,m$  let $e_{\ell m}$ be the oriented edge of $G$ whose image in $X$ is $\eta_{\ell m}$, and so $e_{\ell m}$ has initial vertex~$q_{\ell}$. Let $\epsilon_{\ell m}$ be the initial quarter of the oriented $1$-simplex $e_{\ell m}$, and so as $\ell$, $m$ vary the oriented arcs $\epsilon_{\ell m}$ are pairwise disjoint except possibly at their initial points. Picking a base point for $G$ and paths to $q_{\ell}$ for each~$\ell$, and rechoosing $A_\ell$ in its conjugacy class if necessary, we determine an isomorphism $A_\ell \approx \pi_1(H_\ell,q_\ell)$. An arbitrary $\Psi \in \Aut^{M_1}(A_1) \oplus \cdots \oplus \Aut^{M_L}(A_L)$ is determined by choosing $\Psi_{\ell m} \in \Aut(A_\ell)$ for each $\ell, m$, subject to the requirement that for fixed $\ell$ all of the $\Psi_{\ell m}$ represent the same element of $\Out(A_\ell)$. This data is realized by a homotopy equivalence $f \from G \to G$ as follows. First, $f$ is the identity on $G \setminus \bigl( (\bigcup_{\ell} H_\ell ) \union \bigcup_{\ell m} \epsilon_{\ell m}) \bigr)$. Next, for each $\ell$, $f$ is the identity on $\epsilon_{\ell 1}$ and $f \restrict H_\ell \from (H_\ell,q_\ell) \to (H_\ell,q_\ell)$ is chosen to represent $\Psi_{\ell 1}$. Finally, for each $\ell$ and each $m=2,\ldots,M_\ell$, the automorphisms $\Psi_{\ell 1}$ and $\Psi_{\ell m}$ differ by an inner automorphism of $A_\ell$ which is represented by a closed path $\gamma_{\ell m}$ in $H_\ell$ based at $q_\ell$, and $f(\epsilon_{\ell m}) = \gamma_{\ell m} \epsilon_{\ell m}$. The element $\theta \in \Out(F)$ defined by the homotopy equivalence $f \from G \to G$ is clearly in $\Stab^f[T]$, and by construction $\Theta(\theta) = \Psi$. 

This completes the proof of Lemma~\ref{LemmaCVKT}.

\subsection{Retracting $\CVK$ to $\CVK^T$}
\label{SectionSubcomplexRetract}

To prove Theorem~\ref{TheoremFreeSplitting}, by combining Corollary~\ref{CorollarySubgroupComplex}, Lemma~\ref{LemmaCoarseRetract}, and Lemma~\ref{LemmaCVKT} it suffices to prove:

\begin{theorem} \label{retraction for splittings}  For any free splitting $F_n \act T$, there exists a coarse Lipschitz retraction of $0$-skeleta $R \from \CVK_0 \mapsto \CVK_0^T$.
\end{theorem}

Throughout this section we will freely use the ``Notational conventions for $T$ and $\CVK^T$'' established in Section~\ref{SectionFreeSplittingSubcomplex}.

\smallskip

Suppose that we are given the following data:
\begin{description}
\item[(D1)] An action $F_n \act \bigcup_{V \in \V^\nt} S^V$ on a forest having one component $S^V$ for each $V \in \V^\nt$, so that $S^V$ is stabilized by $\Stab[V]$ and the action $\Stab[V] \act S^V$ is properly discontinuous and minimal.
\item[(D2)] A point $Q_{\ell m} \in S^{V_\ell}$, for each $\ell = 1,\ldots,L$ and each $m=1,\ldots,M_\ell$.
\end{description}
From (D1) and (D2) we construct a $0$-simplex $[S] \in \CVK_0^T$ using the exact method used in the proof of Lemma~\ref{LemmaCVKT} for verifying surjectivity of the map $\I$ on $0$-skeleta. As a graph, $S$ is the disjoint union of the co-edge forest $\wh T$ and the forest $\bigcup_{V \in \V^\nt} S^V$, modulo the identification of each valence~$1$ vertex $\wh V_{\ell m g} \in \wh T$ with the point $g \cdot Q_{\ell m} \in S^{g \cdot V_\ell}$. The actions of $F_n$ on $\wh T$ and on $\bigcup_{V \in \V^\nt} S^V$ are clearly consistent with these identifications, and so they combine to induce a simplicial action $F_n \act S$. There is an equivariant simplicial map $f \from S \to T$ induced by the map that takes $S^V$ to $V \in \V^\nt \subset T$ and by the natural quotient map $\wh T \to T$. Since the pre-image of each point of $T$ is a point or a tree in $S$, the graph $S$ is a tree. The action $F_n \act S$ is minimal because of minimality of the actions $\Stab[V] \act S^V$ and $F_n \act T$. By construction the action $F_n \act S$ is proper, and $f \from S \to T$ is a tight semiconjugacy, and so we have the desired $0$-simplex~$[S]$.

\paragraph{Defining the map $R \from \CVK_0 \to \CVK_0^T$.} Given an arbitrary $0$-simplex $[S'] \in \CVK$ represented by a proper, free splitting $F_n \act S'$, for each $V \in \V^\nt$ let $S^V \subset S'$ be the minimal subtree for the restricted action $\Stab(V) \act S'$. These subtrees $S^V$ need not be pairwise disjoint in~$S'$. Pull them apart by forming their disjoint union over $\V^\nt$, resulting in a forest $\disjunion_{V \in \V^\nt} S^V$ equipped with an action of $F_n$ that evidently satisfies data condition~(D1) --- formally one may identify this disjoint union with the set of ordered pairs $(x,V) \in S' \cross \V^\nt$ such that $x \in S^V$. 

The central idea of the proof is a method for describing the (D2) data of points $Q_{\ell m} \in S^{V_\ell}$ that is needed in order to complete with the definition of $R[S']$. Choose once and for all a base $0$-simplex $[S_0] \in \CVK_0^T$, represented by a proper free splitting $F_n \act S_0$ with tight semiconjugacy $S_0 \mapsto T$. Each of the oriented edges $\ti\eta_{\ell m}$ in $T$ has a unique pre-image in $S_0$ that we denote $e^0_{\ell m}$.  Choose once and for all, for each $\ell$ and $m$, a ray $\gamma^0_{\ell m}$ in $S_0$ with initial oriented edge $e^0_{\ell m}$. Let $\xi_{\ell m} \in \bdy F_n$ be the ideal endpoint of $\gamma_{\ell m}$. Note that $\xi_{\ell m} \not\in \bdy A_\ell$ because $e_{\ell m}$ points out of~$S_0^{V_\ell}$. To specify the data (D2) for $R[S']$, for each $\ell,m$ let $Q'_{\ell m}$ be the point in $S^{V_\ell}$ closest to $\xi_{\ell m}$ --- that is, defining $\gamma_{\ell m}(S')$ to be the unique ray in $S'$ with ideal endpoint $\xi_{\ell m}$ whose intersection with $S^{V_\ell}$ equals its base point, let $Q'_{\ell m}$ be the base point. In the disjoint union $\disjunion_{V \in \V^\nt} S^V$, let $Q_{\ell m} \in S^{V_\ell}$ be the point corresponding to $Q'_{\ell m}$ (formally $Q_{\ell m}$ is the ordered pair $(Q'_{\ell m},V^\ell)$). This completes the data specification~(D2), and we may now complete the definition of $R[S']$ as above.

\paragraph{Example.}  Fix a basis  $F_n = \<a_1,\ldots,a_{n-1},a_n\>$ and identify $F_{n-1}=\<a_1,\ldots,a_{n-1}\> \subgroup F_n$ as usual. Let $R_n$ be the rank $n$ rose with oriented edges $e_1,\ldots,e_n$ corresponding to $a_1,\ldots,a_n$, and identify $R_{n-1} = e_1\union\cdots\union e_{n-1} \subset R_n$. Equip the universal cover $\wt R_n$ with an $F_n$ action using a lift $P \in \wt R_n$ of the unique vertex $p \in R_n$ as a basepoint. Let $\wt R_{n-1}$ be the component of the  full pre-image of $R_{n-1}$ that contains $P$. Let $T$ be the free splitting obtained from $\wt R_n$ by collapsing to a point each translate of $\wt R_{n-1}$; let $V$ be the point to which $\wt R_{n-1}$ itself collapses. The collapse map $\wt R_n \to T$ is a tight semiconjugacy, and we choose the base simplex $[S_0]$ of $\CVK_0^T$ to be represented by the proper free splitting $F_n \act \wt R_n$. The quotient graph of groups $X = T/F_n$ is obtained from $R_n$ by collapsing $R_{n-1}$ to a single point $v$, and having a single edge identified with $e_n$. In $X$ there are two oriented edges: $\eta_1$ which is $e_n$ with positive orientation, and $\eta_2$ with opposite orientation. The action of $a_n$ on $T$ has an axis which passes through $V$; let $\ti \eta_1$, $\ti\eta_2$ be the lifts of $\eta_1,\eta_2$ on that axis with initial vertex~$V$. The action of $a_n$ on $\wt R_n$ has an axis which touches $\wt R_{n-1}$ uniquely at~$P$, and this axis is a union of two rays $\gamma^0_1$, $\gamma^0_2$ intersecting at $P$, so that the initial oriented edge $e^0_i$ of $\gamma^0_i$ maps to $\ti\eta_i$ under the collapse map $\wt R_n \to T$. With these choices, the ideal endpoints $\xi_1,\xi_2$ of $\gamma^0_1,\gamma^0_2$ are the attracting and repelling fixed points, respectively, for the action of $a_n$ on~$\bdy F_n$. For any proper free splitting $F_n \act S'$, letting $S^V \subset S'$ denote the minimal subtree for the action of $F_{n-1}$ as usual, and letting $L^{a_n} \subset S'$ denote the oriented axis of~$a_n$, if $L \intersect S^V$ is empty or a single point then $Q'_1=Q'_2$ is the point of $S^V$ closest to $L$, and otherwise $Q'_1, Q'_2$ are the terminal and initial points, respectively, of the oriented arc $L \intersect S^V$.

\paragraph{The map $R$ is a retraction.} We prove that $R$ fixes each $0$-simplex of $\CVK_0^T$, using induction on distance from $[S_0]$ in the $1$-skeleton of $\CVK^T$, starting with $R[S_0]=[S_0]$ which is clear by construction. Consider an oriented $1$-simplex of the form $\Sigma(G_1,E) \subset \CVK^T$, with initial $0$-simplex $[S_1]=[G_1]$ and terminal $0$-simplex $[S_2]=[G_2]$ where $G_2=G_1/E$ and $E \subset G_1$ is a nontrivial natural subforest. The induction step reduces to proving that $R[S_1]=[S_1]$ if and only if $R[S_2]=[S_2]$; the ``only if'' direction is used when $[S_1]$ is closest to $[S_0]$, the ``if'' direction when $[S_2]$ is closest.

Let $\wt E \subset \wt G_1 = S_1$ be the total lift of $E$, and so $S_2 = \wt G_2$ is obtained from $S_1$ by collapsing to a point each component of the forest $\wt E$. Letting $g \from S_1 \to S_2$ be the collapse map, and letting $f_i \from S_i \to T$ be tight semiconjugacies, we note that $f_2 \composed g = f_1$. For each $i=1,2$ and each $\ell,m$, consider the unique oriented edge $e^i_{\ell m} \subset S_i$ for which $f_i(e^i_{\ell m}) = \ti\eta_{\ell m}$, and consider also the ray $\gamma^i_{\ell m} \subset S_i$ with ideal endpoint $\xi_{\ell m}$ that intersects $S_i^{V_\ell}$ solely in its base point. Note that the trees $S_i^{V_\ell}$ are already pairwise disjoint in $S_i$ as required for the (D1) data, and that the tree $S_i$ can then be reconstructed using the (D2) data where $Q^i_{\ell m}$ is the initial point of $e^i_{\ell m}$. This shows that the equation $R[S_i]=S_i$ is equivalent to the statement that $e^i_{\ell m}$ is the initial edge of the ray $\gamma^i_{\ell m}$. It therefore remains to prove that $e^1_{\ell m}$ is the initial oriented edge of $\gamma^1_{\ell m}$ if and only if $e^2_{\ell m}$ is the initial oriented edge of $\gamma^2_{\ell m}$.

Let $e^i \subset \gamma^i_{\ell m}$ be the initial oriented edges. By Lemma~\ref{LemmaCollapse}, the image $g(\gamma^1_{\ell m}) \subset S_2$ is a ray. Since $g \from S_1 \to S_2$ is an equivariant quasi-isometry, the ideal endpoint of $g(\gamma^1_{\ell m})$ is still $\xi_{\ell m}$. Since $e^1 \not\subset \bigcup_{V \in \V^\nt} S^V$, we have $e^1 \not\subset \wt E$, in other words $e^1$ is not collapsed by $g$. It follows that $g(e^1)$ is the initial oriented edge of $g(\gamma^1_{\ell m})$. Also, since $g(S_1^{V_\ell}) = S_2^{V_\ell}$, it follows that the initial endpoint of $g(e^1)$ is the unique point in which $g(\gamma^1_{\ell m})$ intersects $S_2^{V_\ell}$, and so $g(\gamma^1_{\ell m}) = \gamma^2_{\ell m}$ and $g(e^1) = e^2$. Since $f_2 \composed g = f_1$, it follows that $f_1(e^1)=\ti\eta_{\ell m}$ if and only if $f_2(e^2) = \ti\eta_{\ell m}$. Since $f_1,f_2$ are injective over each edge of $T$, it follows that $e^1 = e^1_{\ell m}$ if and only if $e^2 = e^2_{\ell m}$.

\paragraph{The Lipschitz constant is $1$.} Consider $0$-simplices $[S'_i]=[G'_i] \in \CVK_0$  ($i=1,2$) represented by proper free splittings $F \act S'_i$ with quotient graphs of groups $G'_i$, and suppose that $[S'_1]$, $[S'_2]$ are endpoints of a $1$-simplex in $\CVK$. To show that the map $R$ has Lipschitz constant~$1$ it suffices to show that $[S_1]=R[S'_1]$, $[S_2]=R[S'_2]$ are either equal or are endpoints of a $1$-simplex in~$\CVK^T$. 

Up to reordering we have $G'_2 = G'_1/E'$ for some nontrivial natural subforest $E' \subset G'_1$. Letting $\wt E' \subset \wt G'_1 = S'_1$ be the full preimage, the quotient map $G'_1 \mapsto G'_2$ lifts to an equivariant simplicial map $\pi' \from S'_1 \to S'_2$ that collapses each component of $\wt E'$ to a point. In particular, $\pi'$ induces the identity map on $\bdy F_n$. For each $V \in \V^\nt$, letting $S_1^V \subset S'_1$, $S_2^V \subset S'_2$ denote the $\Stab(V)$ minimal subtrees, Lemma~\ref{LemmaCollapse} implies that $\pi'(S_1^V)=S_2^V$. Define $\wt E^V = \wt E' \intersect S_1^V$, and note that the restricted map $\pi^V = (\pi' \restrict S_1^V) \from S_1^V \to S_2^V$ collapses each component of $\wt E^V$ to a point. After forming the disjoint union of the trees $S_i^V$ over $\V^\nt$ as formally described earlier, and forming the disjunion union of the subforests $\wt E^V$ over $\V^\nt$ --- formally described as those pairs $(x,V)$ where $x \in \wt E^V$, $V \in \V^\nt$ --- the collection of maps $\pi^V$ induces an $F$-equivariant map of forests $\Pi \from \disjunion_{V \in \V^\nt} S_1^V \to \disjunion_{V \in \V^\nt} S_2^V$ that collapses to a point each component of the forest $\wt E = \disjunion_{V \in \V^\nt} \wt E^V$. 

We next extend the collapse map $\Pi$ over the attachments of the co-edge forest $\wh T$. For each $i=1,2$ and each $\ell,m$, let $\gamma^i_{\ell m} \subset S'_i$ be the unique ray with ideal endpoint $\xi_{\ell m}$ intersecting $S_i^{V_\ell}$ solely in its base points $Q^{\prime i}_{\ell m}$. By Lemma~\ref{LemmaCollapse} it also follows that $\pi'(\gamma^1_{\ell m})$ is a ray in $S'_2$. Its ideal endpoint is still $\xi_{\ell m}$. Letting $\alpha \subset \gamma^1_{\ell m}$ be the largest initial subsegment of $\gamma^1_{\ell m}$ such that $\pi'(\alpha)$ is a point, it follows that $\pi'(Q^{\prime 1}_{\ell m}) = \pi'(\alpha)$ is the initial point of the ray $\pi'(\gamma^1_{\ell m})$, that initial point is contained in $S_2^{V_\ell}$, and that initial point is the sole point of intersection of the ray $\pi'(\gamma^1_{\ell m})$ with $S_2^{V_\ell}$; it follows that $\pi'(\gamma^1_{\ell m}) = \gamma^2_{\ell m}$ and that $\pi'(Q{\prime 1}_{\ell m}) = Q{\prime 2}_{\ell m}$. After pulling things apart, it follows that $\Pi(Q^1_{\ell m}) = Q^2_{\ell m}$. Since, in forming $S_i$, the valence~$1$ vertex $V_{\ell m} \in \wh T$ is attached to $Q^i_{\ell m} \in \disjunion S_i^V$, and since and the attachments are extended equivariantly, it follows that $\Pi$ extends to an equivariant map $\Pi \from S_1 \to S_2$. 

By construction, $\Pi$ is the equivariant quotient map that collapses to a point each component of the forest $\wt E \subset \disjunion S_1^V \subset S_1$, and so $\Pi$ descends to a map of quotient marked graphs $G_1 = S_1 / F_n \xrightarrow{\pi} S_2 / F_n = G_2$ that collapses to a point each component of $E = \wt E / B_n \subset G_1$. If $E$ contains no natural edge then the map $\pi$ is homotopic to a homeomorphism and $[S_1]=[S_2]$. Otherwise, letting $\wh E \subset E$ be the union of natural edges in $E$, the map $\pi \from G_1 \to G_2$ is homotopic to a map which collapses each component of $\wh E$ to a point, and so there is a $1$-simplex $\Sigma(G_1,\wh E)$ with initial vertex $[S_1]$ and terminal vertex $[S_2]$. This completes the proof of Theorem~\ref{retraction for splittings}.

\subparagraph{Remark.} Our definition of $R \from \CVK_0 \to \CVK^T_0$ depends on a fairly arbitrary choice of points $\xi_{\ell m} \in \partial F_n$. It is tempting to make more canonical choices using one of various other approaches, based on graph theoretic or metric considerations, or based on surgery arguments using Hatcher's sphere complex \cite{Hatcher:HomStability}. Indeed, we tried several such approaches but none of them gave Lipschitz retractions. As an illustration, we return to the example that follows the definition of $R$, using a construction that arises naturally from the point of view of the sphere complex. If $[S'] \in \CVK^T$ then $S^V \intersect a_n \cdot S^V = \emptyset$ from which it follows that $Q'_1$ is the point of $S^V$ closest to $a_n \cdot S^V$ and $Q'_2$ is the point of $S^V$ closest to $a^\inv_n \cdot S^V$. Otherwise, $S^V \intersect a_n S^V$ is a finite tree $\tau$ (which has a natural description in sphere complex language), and we tried choosing $Q'_1$ to be the centroid of $\tau$, and similarly $Q'_2$ to be the centroid of the tree $a_n^\inv(\tau) = S^V \intersect a_n^\inv \cdot S^V$. This approach defines a retraction of $\CVK$ to $\CVK^T$ that one can show is not Lipschitz. We also considered approaches which somehow pick out a certain valence $1$ vertex of $\tau$ (corresponding to picking out an ``innermost disc'' in sphere complex language), but no such approach that we tried yields a Lipschitz retraction.

\section{Stabilizers of free factor systems in $\Out(F_n)$}
\label{SectionTandA}

In this section we prove Theorem~\ref{TheoremCoindex}, regarding coarse Lipschitz retraction versus distortion for the subgroup $\Stab(\F) \subgroup \Out(F_n)$ stabilizing a free factor system $\F$ of $F_n$. 

In section~\ref{SectionCorankOne} we prove Theorem~\ref{TheoremCoindex} item~\pref{ItemFFSCoindexOne}; this is the \emph{only} place in Section~\ref{SectionTandA} where we will appeal to the results of either of Sections~\ref{SectionAutNondistortion} or~\ref{SectionTreeStabilizers}.
In section~\ref{SectionFFSubcomplex} we study a subcomplex $\CVK^\F_n \subset \CVK_n$ corresponding to free factor system~$\F$, to which Corollary~\ref{CorollarySubgroupComplex} can be applied. In section~\ref{SectionCorank2} we prove distortion of $\Stab(\F)$ in $\Out(F_n)$ when $\coindex(\F) \ge 2$, by proving distortion of $\CVK^\F_n$ in $\CVK_n$. 

\subsection{Proof of Theorem~\ref{TheoremCoindex} item~\pref{ItemFFSCoindexOne}.}
\label{SectionCorankOne}

Suppose that $\F$ is a free factor system of coindex~$1$ in~$F_n$. Choose a marked graph $G$ in which $\F$ is realized by a core subgraph $H$ such that each component of $H$ is a rose  and such that the  complement of $H$ is a single oriented edge $E$. Define $T$ to be the free splitting obtained from the universal cover $\wt G$  of $G$ by collapsing each component of the full pre-image $\wt H$ of $H$ to  a point.  In light of Theorem~\ref{TheoremFreeSplitting}, it suffices to show that  $\Stab(\F)=\Stab[T]$.  

The inclusion $\Stab[T] \subset \Stab(\F)$ follows from $\F(T) = \F$, which is immediate from the definition of $T$. For the reverse inclusion, suppose that $\phi(\F) = \F$.   Choose a homotopy equivalence $f :G \to G$ that represents $\phi$ and that preserves $H$.  By Corollary~3.2.2 of \BookOne\  there are (possibly trivial) paths $u,v \in H$ such that either $f(E) = u E v$ or $f(E) = u E^{-1} v$.  In either case, $f$ lifts to  $\ti f: \wt G \to \wt G$ that preserves $\wt H$, inducing a bijection on the set of components of $\wt H$,  and such that the image of any edge in the complement of $\wt H$ crosses exactly one edge in the complement of $\wt H$.  Thus $\ti f$ induces a conjugacy between the given action $F_n \act T$ and the precomposition of that action by an automorphism $\Phi \in \Aut(F_n)$ that represents $\phi$, and it follows that $\phi \in \Stab[T]$.

\subsection{Free factor systems, their stabilizers, and their subcomplexes.} 
\label{SectionFFSubcomplex}

Fix a proper free factor system $\F=\{[A_1],\ldots,[A_k]\}$ of $F_n$ and define a flag subcomplex $\CVK^\F_n$ of $\CVK_n$ by requiring that a $0$-simplex $[(G,\rho)] \in \CVK_n$ is in $\CVK_n^\F$ if and only if   $G$ has a core subgraph $H$ satisfying $\F=[H]$. 

\begin{lemma} \label{LemmaCVKA}
$\CVK_n^\F$ is a connected subcomplex of $\CVK_n$, preserved by $\Stab(\F)$, and with finitely many orbits of cells under the action of $\Stab(\F)$. It follows that $\Stab(\F)$ is finitely generated, and that for any $0$-simplex $[(G,\rho)]$ of $\CVK_n^\F$ the orbit map $\Stab(\F) \mapsto \CVK_n^\F$ taking $\phi$ to $[(G,\rho)]\phi$ is a quasi-isometry.
\end{lemma}

\begin{proof} The second sentence follows from the first, by applying the \MilnorSvarc\ lemma.

\smallskip

To prove that $\Stab(\F)$ leaves $\CVK_n^\F$ invariant, it suffices to prove that its $0$-simplices are invariant. Consider a $0$-simplex $[(G,\rho)]$ in $\CVK_n^\F$ with  subgraph $H \subset G$ with $[H]=\F$ --- to be more explicit, under the isomorphism $\rho_* \from F_n = \pi_1(R_n) \to \pi_1(G)$ we have $\rho_*(\F)=[H]$. Consider $\phi \in \Stab(\F)$ represented by a self-homotopy equivalence $\Phi \from R_n \to R_n$, so $\Phi_*(\F) = \F$ and $(G,\rho) \phi = (G,\rho\Phi)$. We have $(\rho \Phi)_*(\F) = \rho_* \Phi_*(\F) = \rho_*(\F) = [H]$. It follows that the subgraph $H$ still represents the free factor system $\F$ with respect to the marking $\rho \Phi$ on the graph $G$, proving that $[(G,\rho)] \phi \in \CVK_n^\F$.

\smallskip

To prove that $\CVK_n^\F$ has only finitely many $\Stab(\F)$ orbits of cells, again it suffices to prove this for 0-cells. Let $\F=\{[A_1],\ldots,[A_k]\}$. Consider the set of triples $(G,\rho,H)$ such that $[(G,\rho)]$ is a $0$-simplex of $\CVK_n$ and $H \subset G$ is a subgraph with $k$ noncontractible components, having ranks equal to $\rank(A_1),\ldots,\rank(A_k)$, respectively. Define an equivalence relation on this set of triples, where $(G,\rho,H)$ is equivalent to $(G',\rho',H')$ if there exists an isometry $h \from G \to G'$ such that $h(H)=H'$ and $h \composed \rho$ is homotopic to $\rho'$. The group $\Out(F_n)$ acts on such equivalence classes just as it does on marked graphs themselves, and there are only finitely many orbits of this action, by the same proof that $\CVK_n$ has only finitely many $\Out(F_n)$ $0$-simplex orbits. It therefore suffices to prove that if $(G,\rho)$ and $(G',\rho')$ represent $0$-simplices of $\CVK_n^\F$, if $H \subset G$, $H' \subset G'$ are subgraphs with $[H]=[H']=\F$, and if the triples $(G,\rho,H)$, $(G',\rho',H')$ are in the same $\Out(F_n)$ orbit, then they are in the same $\Stab(\F)$ orbit. Choose $\phi \in \Out(F_n)$ so that $(G,\rho,H)\phi$ is equivalent to $(G',\rho',H')$, so there exists an isometry $h \from G \to G'$ and there exists a homotopy equivalence $\Phi \from R_n \to R_n$ representing $\phi$ such that $h(H)=H'$ and $h \composed \rho$ is homotopic to $\rho' \composed \Phi$. We have 
$$\phi_*\F = \phi_* \rho^\inv_*[\pi_1 H] 
   = (\rho')^\inv_* h_*[\pi_1 H] 
   = (\rho')^\inv_* [\pi_1 H'] 
   = \F
$$
and so $\phi \in \Stab(\F)$.

\medskip

We prove connectivity of the subcomplex $\CVK_n^\F$ using a Stallings fold argument. Choose a base vertex $[G'] \in \CVK_n^\F$ so that each component of the core subgraph $H' \subset G'$ realizing $\F$ is a rose and so that all vertices of $G'$ are contained in $H'$ (we are suppressing $F_n$-markings from the notation for marked graphs). Given an arbitrary vertex $[G] \in \CVK_n^\F$ we must find an edge path in $\CVK_n^\F$ connecting $[G']$ to $[G]$. By successively collapsing to a point each edge of $G \setminus H$ having exactly one endpoint in $H$, and collapsing to a point each edge in $H$ with distinct endpoints, we may assume that each component of $H$ is a rose with one vertex and that all vertices of $G$ are contained in $H$. By moving each component of $H$ through the connected autre espace of the corresponding component of~$\F$, we may assume that the components of $H$ and $H'$ corresponding to the same component $[A_i] \in \F$ are equivalent as pointed marked $A_i$ graphs. It follows that there exists a homotopy equivalence $h \from (G,H) \to (G',H')$ that preserves $F_n$-markings, takes vertices to vertices, and restricts to a homeomorphism $H \mapsto H'$. Furthermore, the restriction of $h$ to each edge $e$ of $G \setminus H$ is homotopically nontrivial rel endpoints, and so by homotoping $h \restrict e$ rel endpoints we may assume that $h \restrict e$ is an immersion.

Now factor $h$ as a composition of Stallings folds
$$G = G_0 \xrightarrow{h_1} G_1 \to\cdots\to G_{k-1} \xrightarrow{h_k} G_k=G'
$$
and recall the following properties: each $G_i$ is a marked graph and each $h_i$ preserves markings; for each $i=0,\ldots,k-1$, denoting $g_i = h_k \composed\cdots\composed h_{i+1} \from G_i \to G'$, there exist oriented edges $e_{i1},e_{i2} \subset G_i$ with the same initial vertex, and there exist initial segments $\eta_{i1} \subset g_{i1}$ such that $g_i(\eta_{i1})$ and $g_i(\eta_{i2})$ are equal as edge paths in $G'$, the initial segments $\eta_{i1},\eta_{i2}$ are maximal with respect to this property, and the map $h_i$ identifies the segments $\eta_{i1}$, $\eta_{i2}$ respecting $g_i$-images. Since $h$ restricts to a homeomorphism from $H$ to $H'$ it follows that in each $G_i$ there is a subgraph $H_i$ such that $h_i$ restricts to a homeomorphism from $H_{i-1}$ to $H_i$. It follows that each $H_i$ is a core subgraph representing $\F$ and so $[G_i] \in \CVK_n^\F$ for each $i$.

Proceeding by induction on $k$, we may assume that $[G_1]$, $[G']$ are connected by an edge path in $\CVK_n^\F$, and it remains to find an edge path connecting $[G_0]$ and $[G_1]$. The fold map $h_1 \from G_0 \to G_1$ may be factored as $G_0 \xrightarrow{q'} G' \xrightarrow{q''} G_1$ where $q'$ identifies initial segments of $\eta_{i1}$ and $\eta_{i2}$ which are proper in $e_{i1},e_{i2}$, respectively, and $q''$ identifies the remaining segments of $\eta_{i1},\eta_{i2}$. Since $h_1$ restricts to a homeomorphism from $H_0$ to $H_1$ it follows that $q'$ restricts to a homeomorphism from $H_0$ to $H'$ and $q''$ restricts to a homeomorphism from $H'$ to $H_1$, and again it follows that $H'$ is a core subgraph representing $\F$ and so $[G'] \in \CVK_n^\F$. There is a forest collapse $G' \to G_0$ which collapses to a point the natural segment $q'(\eta_{i1})=q'(\eta_{i2}) \subset G'$, and so $[G],[G']$ are endpoints of an edge of $\CVK_n^\F$. There is also a forest collapse $G' \to G_1$ which collapses to a point the union of the two segments $q'(e_{i1}-\eta_{i1})$ and $q'(e_{i2}-\eta_{i2})$ (one or both of which may be proper segments of natural edges), and it follows that $[G']$, $[G_1]$ are either equal or connected by an edge. 
 \end{proof}

\subsection{Stabilizers of free factor systems of coindex $\ge 2$ in $\Out(F_n)$}
\label{SectionCorank2}

In this section we prove the remaining half of Theorem~\ref{TheoremCoindex}: if $\F$ is a free factor system of $F_n$ with $\coindex(\F) \ge 2$ then $\Stab(\F)$ is distorted in $\Out(F_n)$. We give full details for the special case of a connected free factor system considered in Theorem~\ref{TheoremCorank}: if $A$ is a nontrivial free factor of $F_n$ of rank $\le n-2$ then $\Stab[A]$ is distorted. Afterwards we discuss the changes needed to handle the disconnected case. 

\paragraph{Case 1: Connected free factor systems.} Fix $r$ with $1 \le r \le n-2$, and fix $A$ to be the free factor of $F_n = \<a_1,\ldots,a_n\>$ given by $A = \<a_1,\ldots,a_r\>$. We shall prove exponential distortion of $\Stab[A]$ by exhibiting a sequence $\phi_k \in \Stab[A]$ whose word length in $\Out(F_n)$ has a linear upper bound in terms of $k$, but whose word length in $\Stab[A]$ has an exponential lower bound. This will suffice for proving exponential distortion of every connected rank~$r$ free factor system, because each is the image of $[A]$ under the action of some element of $\Out(F_n)$ and so its stabilizer is the image of $\Stab[A]$ under conjugation by that element.

Let $m=r+1$ so $2 \le m \le n-1$. Let $\theta \in \Out(F_n)$ be represented by any homotopy equivalence $\Theta \from R_n \to R_n$ which is the identity on the subrose $R_n \setminus R_{m} = e_{m+1} \union\cdots\union e_n$, and restricts to an exponentially growing train track map on $R_{m}$. For concreteness we may take
\begin{align*}
\Theta(e_1) &= e_1 e_{m} \\
\Theta(e_i) &= e_{i-1} \quad\text{if $2 \le i \le m$} \\
\Theta(e_i) &= e_i \quad\text{if $m+1 \le i \le n$} 
\end{align*}
which is easily seen to be a homotopy equivalence since $m \ge 2$. This map $\Theta$ is a train track map, meaning that for all $k \ge 0$ the restriction of $\Theta^k$ to each edge is an edge path without cancellation; this is true because $\Theta$ is a positive map, respecting orientation of edges. Also, by an application of the Perron-Frobenius theorem \cite{Hawkins:PerronFrobenius}, the action of $\Theta$ on $R_{m}$ is exponentially growing, meaning that there exist constants $C > 0$, $b > 1$ such that for all $i,j=1,\ldots,m$ and all $k \ge 1$ the number of occurrences of $e_j$ in the edge path $\Theta^k(e_i)$ exceeds $C b^k$; again this uses that $m \ge 2$.

Recall from the introduction the sequence $\phi_k \in \Out(F_n)$ represented by homotopy equivalences $\Phi_k \from R_n \to R_n$ defined by 
\begin{align*}
\Phi_k(e_i) &= e_i \quad\text{if $i < n$} \\
\Phi_k(e_n) &= e_n  u_k
\end{align*}
where $u_k = \Theta^k(e_1)$. From this expression we clearly have $\phi_k \in \Stab[A]$. Letting $\overline\Theta \from R_n \to R_n$ be a homotopy inverse for $\Theta$ which maps $R_{m}$ to itself and is the identity on $R_n \setminus R_{m}$, the homotopy equivalence $\Theta^k \Phi \overline\Theta^k \from R_n \to R_n$ represents $\phi_k$. We therefore have $\phi_k = \theta^k \phi_0 \theta^{-k} \in \Out(F_n)$, an expression which patently demonstrates that the word length of $\phi_k$ in the group $\Out(F_n)$ has a linear upper bound in~$k$. It remains to prove that the word length of $\phi_k$ in the group $\Stab[A]$ has an exponential lower bound.

For future reference we note that $\Theta^k \Phi \overline\Theta^k$ preserves the subrose $R_{m}$ and restricts to a map homotopic to the identity on $R_{m}$ because it is the result of a homotopy conjugation of the identity map $\Phi \restrict R_{m}$. Denoting $B = \<a_1,\ldots,a_{m}\>$ and noting that $[B]=[R_{m}]$, we have the following generalization of the fact that $\phi_k \in \Stab[A]$:
\begin{itemize}
\item For any free factor system $\F$ such that $\F \sqsubset [B]$ we have $\phi_k(\F)=\F$, and so $\phi_k \in \Stab(\F)$.
\end{itemize}

Recall that a vertex of the subcomplex $\CVK_n^{[A]} \subset \CVK_n$ is represented by marked graph $G$ in which $[A] = [H]$ for some core subgraph $H \subset G$.  Given such a $G$ and $H$, let $S$ be the minimal $B$-subtree of the universal cover $\ti G$ and let $K$ be the marked $B$-graph that is the  quotient of $S$  by $B$. The simplicial structure induced on $K$ by the natural structure on $G$ need not be the same as the natural structure on $K$.  Since $A \subgroup B$ and since $A$ is realized by a core subgraph in $G$, the minimal $A$-subtree $S_A$ of $\ti G$ is contained in $S$ and is disjoint from each of its non-identical translates by elements of $B$. Thus, the quotient graph $K_A$ obtained from $S_A$ by dividing out by the action of $A$ is a core subgraph $K_A \subset K$ with $[K_A] = [A]$. As the difference in rank between $A$ and $B$ is $1$, the complement of $K_A$ in $K$ is either a single natural edge with both endpoints in $K_A$ or a loop that is disjoint from $K_A$ union a natural edge that connects the loop to $K_A$. In the former case we let $E$ be the unique natural edge in the complement of $K_A$ and in the latter case we let $E$ be the loop in the complement of~$K_A$.

For   any nontrivial conjugacy class $\con$ in $F_n$ that is not represented by an element of $B$,   we shall define a non-negative integer
$$i(\con,G)  = i_{A,B}(\con,G) 
$$
where the subscripts $A$ and $B$ are suppressed when they are clear by context. First we define $i_{A,B}(c,G)$ for any $c$ in $\con$ as follows.   Let $L_c$ be the axis in $\ti G$ for the action of $c$. Since $ c \not \in B$, the intersection  of   $L_c$ with $S$ is a finite path in $S$ that projects to a finite path $\mu$ in $K$ whose endpoints need not be natural vertices.  Define $i_{A,B}(c,G)$ to be the number of  times that $\mu$ entirely crosses $E$ in either direction. Then define $i_{A,B}(\con,G)$ to be the supremum, over all $c \in F_n$  in the conjugacy class $\con$, of $i_{A,B}(c,G)$.  

To see that $i_{A,B}(\con,G)$ is finite, note that $i_{A,B}(c^b,G) =i_{A,B}(c,G)$ for all $b \in B$ because $L_{c^b} \cap S$ and $L_c \cap S$ project to the same path in $K$.  
For any given compact subset $X \subset \ti G $ there exist only finitely many elements $c$ in the conjugacy class of $\con$ such that  $L_c \cap X$ is non-trivial.  Choosing $X$ to be a fundamental domain in $S$ for the action of $B$, it follows that  every $B$-orbit in $\gamma$ that has nontrivial intersection with $S$ has a representative that intersects $X$. It follows that $i_{A,B}(\con,G)$ is the maximum value of $i_{A,B}(c,G)$ as $c$ varies over the finitely many elements in the conjugacy class of $\con$ such that $L_c \cap X$ is non-empty.  Since  each $i_{A,B}(c,G)$ is finite  by construction, the maximium value is also finite.    

We note the following properties of $i_{A,B}(\con,G)$:
\begin{description}
\item[Conjugacy invariance] $i_{A,B}(\con,G)$ depends only on the conjugacy classes of $A$, $B$, and the equivalence class of the marked graph~$G$.
\item[Equivariance] $\displaystyle i_{A,B}(\con,G \psi) = i_{\psi[A],\psi[B]}(\psi(\con),G)$ for each $\psi \in \Out(F_n)$.
\end{description}

The quantity $i_{A,B}(\con,G)$ plays the same role in our proof of an exponential lower bound as is played by the function $\ti\alpha(\con)$ in Alibegovi\'c's paper \cite{Alibegovic:translation}, where an infinite cyclic subgroup $\Out(F_n)$ is proved to be undistorted by providing a linear lower bound to the word length of the powers of the generator. The quantity $\ti\alpha(\con)$ is a ``count'' of the maximum number of consecutive fundamental domains of other group elements that occur inside the axes of elements in the conjugacy class $\con$. Lemma~2.4 of \cite{Alibegovic:translation} says that this count barely changes under small moves, namely when $\ti\alpha(\con)$ is replaced by $\ti\alpha(g(\con))$ for a generator $g$ of $\Out(F_n)$. 
 
The following lemma shows that $i_{A,B}(\con,G)$ barely changes under small moves, namely as $G$ moves along an edge in $\CVK_n^{[A]}$.

\begin{lemma}\label{LemmaCount}
For any nontrivial conjugacy class $\con$ in $F_n$ which is not represented by an element of $B$, the map $G \mapsto i_{A,B}(\con,G)$ is an integer valued $2$-Lipschitz function on the vertex set of the complex $\CVK_n^{[A]}$.
\end{lemma}

Accepting this lemma for the moment, we complete the proof of distortion of $\Stab[A]$ by finding an exponential lower bound in $k$ to the word length of $\phi_k$ in the group $\Stab[A]$.

Let $(G_0,\rho_0) = (R_n, Id)  \in \CVK_n^{[A]}$.  Since $A$ is represented by the subrose $R_r=R_{m-1}$ and $B$ by the subrose $R_{m}$, it follows that for any conjugacy class $\con$ in $F_n$ that is not contained in $R_{m}$, the number $i_{A,B}(\con,G_0)$ may be computed as follows: enumerate the maximal subpaths in $R_m$ of the unique circuit in $R_{n}$ that represents $\gamma$, count the number of occurrences of $e_{m}$ (in either direction) in each such subpath, and take the maximum count over all such subpaths.

Let $\con_0$ be the conjugacy class in $F_n$ of $e_n$. Note that $i_{A,B}(\con_0,G_0) = 0$. Let $\con_k = \phi_k(\con_0)$, which is the conjugacy class in $F_n$ represented by
$$\Phi_k(e_n) = \Theta^k \Phi \overline \Theta^k(e_n) = \Theta^k \Phi (e_n) = \Theta^k (e_n e_1) = e_n \Theta^k(e_1)
$$
which is clearly a circuit in $R_n$, because $\Theta \restrict R_{m}$ is a train track map; also, this circuit is not contained in $R_{m}$. Since this circuit has exactly one maximal subpath intersecting the subrose $R_{m}$, namely $\Theta^k(e_1)$, it follows that $i_{A,B}(\con_k,G_0)$ is equal to the number of occurences of $e_{m}$ in the edge path $\Theta^k(e_1)$. As described above after the definition of $\Theta$, this number has an exponential lower bound in $k$. 

It follows that the number
$$\abs{i_{A,B}(\con_0,G_0\phi_k) - i_{A,B}(\con_0,G_0)} = i_{A,B}(\con_0,G_0\phi_k) = i_{\phi_k[A],\phi_k[B]}(\phi_k(\con_0),G_0)=i_{A,B}(\con_k,G_0)
$$
has an exponential lower bound in $k$. Applying Lemma~\ref{LemmaCount}, the length of the shortest path in $\CVK_n^{[A]}$ from $G_0$ to $G_0 \phi_k$ has an exponential lower bound in $k$. Applying Lemma~\ref{LemmaCVKA} and Corollary~\ref{CorollarySubgroupComplex}, the word length of $\phi_k$ has an exponential lower bound in $k$, finishing the proof of distortion of $\Stab[A]$.

\begin{proof}[Proof of Lemma \ref{LemmaCount}]  Consider two marked graphs $G,G'$ representing $0$-simplices in $\CVK_n^{[A]}$. We must show that if these $0$-simplices are connected by a $1$-simplex and if $c$ is contained in the conjugacy class of $\con$ then   $\abs{i(c,G) - i(c,G')} \le 2$. 

 Let $S'$ be the minimal $B'$-subtree of $\wt G'$,  let $K'$ be its quotient   under the action of $B$ and let $K'_A$ be the core subgraph of $K'$ such that $[K'_A] = [A]$.  Let $L'_c$ be the axis for the action of $c$ on $\wt G'$. The projected image of $L_c \cap S$ is a finite path $\mu$ in $K$, and the projected image of $L'_c \cap S'$ is a finite path $\mu'$ in $K'$.  If the complement of $K'_A$ in $K'$ is a single natural edge, denote that edge by $E'$.  Otherwise the complement of $K'_A$ in $K'$ is a loop that we denote $E'$, connected to $K'_A$ by a natural edge.

  We may assume without loss that $G'$ is obtained from $G$ by collapsing a subforest $F \subset G$; let $\pi_F : G \to G'$ be the corresponding quotient map.   Lift $\pi_F :G \to G'$ to the quotient map $\pi_{\wt F} :\wt G \to \wt G'$     that collapses each component of  the full pre-image $\wt F$ of $F$ to a point.  By Lemma~\ref{LemmaCollapse}, $\pi_{\wt F} (S, L_a \cap S) = (S', L'_a \cap S')$.      Letting  $F_K \subset K$ be the forest  that is the quotient of   $\wt F \cap S$ by the action of $B$ and letting $\pi_{F_K}$ be the quotient map that collapses each component of $F_k$ to a point, it follows that $\pi_{F_K} (K, \mu) = (K', \mu')$.   
  
  In what follows, one should  keep in mind that $F_K$ is a union of edges in the induced simplicial structure $K$ but not not necessarily in the natural simplicial structure on $K$.
  
We claim that $E$ is not contained in $F_K$ and moreover that $\pi_{F_K}(E) = E'$. If $E$ is the only natural edge in the complement of $K_A$ then this follows from the fact that $\pi_{F_K}(K_A) = K'_A$ and the fact that the interior of $E'$ is the complement of $K'_A$. If $E$ is a loop that is disjoint from $H$ then $F_K$ does not contain $E$ because $F_K$ is a forest, so $\pi_{F_K}(E)$ is a loop in $K'$ that intersects $K'_A$ in either zero or one points; in zero points then $\pi_{F_K}(E)$ is a loop in the complement of $K'_A$ and if $1$ point then the arc connecting $E$ to $K_A$ was contained in $F_K$ so the interior of $\pi_{F_K}(E)$ is the complement of $K'_A$.  In both cases, $\pi_{F_K}(E) = E'$. This completes the proof of the claim.

        With the possible exception of initial and terminal segments that are properly contained in natural edges, $\mu$ factors as a concatenation of natural edges.  The $\pi_{F_K}$ image of a natural edge $e$ crosses $E'$ if and only if $e = E$ or $e = E^{-1}$ in which case it crosses $E'$ exactly once.   The   $\pi_{F_K}$ image of the initial segment of $E$ crosses $E'$ either zero or one time and similarly for the terminal segment of $E$.  This proves that   $i(c,G)\le  i(c,G') \le i(c, G)+2$ as desired.   
\end{proof}

This completes the proof of the corank~$\ge 2$ case of Theorem~\ref{TheoremCorank}, which is the case of a connected free factor system of co-index~$\ge 2$ in Theorem~\ref{TheoremCoindex}.

\paragraph{Case 2: Free factor systems with $2$ components.} We now describe how to adapt the proof of distortion for stabilizers of connected free factor systems to the case of the stabilizer of a two component free factor system $\F=\{[A_0],[A_1]\}$ such that $I = \coindex(\F) \ge 2$. Our proof will have features that will allow us to apply it also to the case of free factor systems with $\ge 3$ components.

Let $G'$ be a marked graph having rose subgraphs $H'_0,H'_1,H'_2$, an edge $\eta'_0$ oriented from the vertex $v'_0 \in H'_0$ to the vertex $v'_1 \in H'_1$, and an edge $\eta'_1$ oriented from the vertex $v'_2 \in H'_2$ to the vertex $v'_1$, such that $[H'_0]=[A_0]$ and $[H'_1]=[A_1]$. Since $\rank(H'_2) = \coindex(\F) -1 \ge 1$, it follows that the rose $H'_2$ does not degenerate to a single point and so the graph $G'$ has all vertices of valence~$\ge 3$, as required for a marked graph. Collapsing the edge $\eta'_0$ results in a marked graph $G$ and an edge collapse map $G' \mapsto G$ which preserves marking. Under this collapse map, for each $i = 0,1,2$ the subgraph $H'_i$ goes by homeomorphism to a subgraph $H_i \subset G$. Also, the subgraph $H'_0 \union \eta'_0 \union H'_1$ goes by homotopy equivalence to the rose subgraph $H = H_0 \union H_1 \subset G$, and we denote $m = \rank(H) = \rank(H'_0) + \rank(H'_1) \ge 2$; and the edge $\eta'_1$ goes by homeomorphism to an edge $\eta_1 \subset G$ connecting the vertex $v_2 \in H_2$ to the vertex~$v \in H$.

We define the sequence of outer automorphisms $\phi_k \in \Out(F_n)$ by picking representative homotopy equivalences $\Phi_k \from G \to G$ which restrict to the identity on $G \setminus \eta_1 = H \union H_2$ and which satisfy $\Phi_k(\eta_1) = \eta_1 u_k$ for an appropriately chosen sequence of closed paths $u_k$ in~$H$. To be precise, enumerate the edges of $H$ as $e_1,\ldots,e_m$, then let $\theta \in \Out(F_n)$ be represented by a homotopy equivalence $\Theta \from G \to G$ which is the identity on $G \setminus H$ and whose restriction to the rose $H$ is given explicitly by the same train track map used in the connected case, and then use the same sequence $u_k = \Theta{}^k(e_1)$.

The outer automorphism $\phi_k$ preserves the free factor system $\{[H],[H_1]\}$ and restricts to the identity element on each of the groups $\Out(\pi_1(H))$ and $\Out(\pi_1(H_1))$. It follows that $\phi_k$ also preserves any free factor system $\A \sqsubset \{[H],[H_1]\} = \{[H'_0 \union \eta'_0 \union H'_1],[H'_2]\}$, and so $\phi_k \in \Stab(\A)$; this fact will be needed later. For the moment, we apply this fact to the case that $\A=\F=\{[H'_0],[H'_1]\}$ to conclude that $\phi_k \in \Stab(\F)$, as needed.

Noting that $\Phi_k$ is homotopic to $\Theta{}^k \Phi_0 \overline\Theta{}^{k}$, where $\overline\Theta$ is any homotopy inverse for $\Theta$ that preserves $H$ and is the identity on $G \setminus H$, we see that the word length in $\Out(F_n)$ of $\phi_k = \theta^k \phi_0 \theta^{-k}$ has a linear upper bound, as needed.

We would like to get an exponential lower bound for the word length of $\phi_k$ in $\Stab(\F)$ by counting occurrences of the edge $\eta'_0$ in the path $u_k$. This does not quite make sense because $\eta'_0$ and $u_k$ live in different marked graphs $G',G$, but we can relate them by using the edge collapse $G'\mapsto G$ and counting subpaths of $u_k$ which cross between $H_0$ and $H_1$. 

We claim that the number of two-edge subpaths of $u_k$ which cross between $H_0$ and $H_1$ has an exponential lower bound in $k$. To see why, first note that there exists $k_0$ and $j=1,\ldots,m$ such that the edge path $\Theta{}^{k_0}(e_j)$ contains at least two copies of every edge in $H$, and so it contains some edge of $H_0$ which comes before some edge of~$H_1$. By connectivity, $\Theta{}^{k_0}(e_j)$ contains a subpath $\rho = e_{i_0} e_{i_1}$ with $e_{i_0}$ in $H_0$ and $e_{i_1}$ in $H_1$. Next, noting that $\Theta{}^{k-k_0}(e_1)$ contains an exponentially growing number of copies of $e_j$, it follows that $u_k=\Theta{}^k(e_1)$ contains an exponentially growing number of copies of $\rho$, proving the claim.

Now we pull the description of $\phi_k$ back to the marked graph $G'$. The edges of $H'_0 \union H'_1$ can be oriented and listed as $e'_1,\ldots,e'_m$ so that the edge collapse $G' \mapsto G$ takes $e'_i$ to $e_i$. Let $u'_k$ be the unique closed edge path in $G'$ with initial and terminal points at $v'_1$ such that the edge path in $G$ obtained from $u'_k$ by collapsing each copy of $\eta'_0$ is equal to $u_k$. To be precise, $u'_k$ is obtained from $u_k$ by adding prime symbols $\prime$ to the edges in the edge path~$u_k$, inserting a copy of $\eta'_0$ or $\bar \eta'_0$ wherever $u_k$ has a two edge subpath which crosses between $H_0$ and $H_1$, inserting a copy of $\bar \eta'_0$ at the beginning if $u_k$ starts with an edge of $H_0$, and inserting a copy of $\eta'_0$ at the end if $u_k$ ends with an edge of $H_0$. From this description it follows that the number of occurrences of $\eta'_0$ or $\bar \eta'_0$ in $u'_k$ has an exponential lower bound. 

Note that $\phi_k$ is represented by the homotopy equivalence $\Phi'_k \from G' \to G'$ which restricts to the identity on $G' \setminus \eta'_1 = H'_0 \union \eta'_0 \union H'_1 \union H'_2$ and satisfies $\Phi'_k(\eta'_1) = \eta'_1 u'_k$.

\medskip

We are now ready to adapt the proof of Case~1 to proving Case 2. Choose a free factor $B \subgroup F_n$ in the conjugacy class $[H]$. Choose the free factors $A_0,A_1$ in the conjugacy classes $[A_0]=[H_0]$, $[A_1]=[H_1]$ so that $B=A_0 * A_1$. For an arbitrary vertex $[(G^*,\rho^*)] \in \CVK^\F_n$, let $S$ be the minimal $B$-subtree of $\wt G^*$. The minimal $A_0$ and $A_1$ subtrees $S_{A_0},S_{A_1}$ are disjoint from each other and from each of their non-identical translates by elements of $B$, so the quotient graph $K$ obtained from $S_B$ by dividing out by the action of $B$ contains disjoint core subgraphs $K_{A_0}$ and $K_{A_1}$ such that $[K_{A_0}] = A_0$ and $[K_{A_1}] = A_1 $.  Since  the rank of $B$ is the sum of the rank of $A_1$ and the rank of $A_2$, the complement of $K_{A_0} \cup K_{A_1}$ in $K$ is a single natural edge $E$.   Given $c \in F_n$ that is not contained in $B$,   define $i(c,G^*)$ to be the number of times that the projection of $L_c \cap S$ into $K$ completely crosses $E$.   Given a conjugacy class $\con \in F_n$ that is not represented by an element in $B$, define $i(c,G^*)$ to be the maximum of $i(c, G^*)$ over all  elements of $F_n$ representing $\con$.  The proof of Lemma~\ref{LemmaCount} easily adapts to this situation, with the conclusion that if $\con$ is not represented by an element of $B$ then the map $G^* \mapsto i(\con,G^*)$ is an integer valued 2-Lipschitz function on the vertex set of $\CVK^\F_n$.

One last issue is to choose an appropriate conjugacy class $\con_0$. In the connected case we chose $\con_0$ to be represented by the loop edge whose role is played here by $\eta'_1$, but $\eta'_1$ does not form a loop in $G'$. Instead we choose $\con_0$ to be the conjugacy class represented in $G'$ by a loop of the form $\gamma' = \rho' \eta'_1 \sigma' \bar \eta'_1$. We choose $\rho'$ to any loop in $H'_2$. We choose $\sigma'$ to be a loop in $H' = H'_0 \union \eta'_0 \union H'_1$ based at $v'_1$, taking care to avoid cancellation in the path 
$$\Phi'_k(\gamma') = \Phi'_k(\rho') \Phi'_k(\eta'_1) \Phi'_k(\sigma') \Phi'_k(\bar \eta'_1) = \rho' \eta'_1 u'_k \sigma' \bar u'_k \bar \eta'_1
$$
The two places where cancellation can occur is at the concatenation points of $u'_k \sigma'$ and of~$\sigma' \bar u'_k$. Since $u'_k$ is a concatenation of the positively oriented edges $e'_1,\ldots,e'_m$ and copies of $\eta'_0$ and $\bar \eta'_0$, we need only choose $\sigma'$ to start with some positively oriented edge $e'_1,\ldots,e'_m$ and to end with some negatively oriented edge $\bar e'_1,\ldots,\bar e'_m$. For example, taking $e'_{l_0}$ in $H'_0$ and $e'_{l_1}$ in $H'_1$ we may take $\sigma' = e'_{l_1} \bar \eta'_0 e'_{l_0} \eta'_0 \bar e'_{l_1}$. Note that $i(\con_0,G)=2$, coming from the $2$ times that $\sigma'$ crosses $\eta'_0$ or $\bar \eta'_0$. Also, letting $\con_k$ be the conjugacy class represented by the loop $\Phi'_k(\gamma')$, and noting that this loop has a single subpath $u'_k \sigma' \bar u'_k$ in the subgraph $H'$, and that this subpath has an exponentially growing number of occurrences of $\eta'_0$ or $\bar \eta'_0$, it follows that $i(\con_k,T)$ has an exponential lower bound in $k$. 

We therefore have $\abs{i(\con_0,G\phi_k) - i(\con_0,G)} = i(\con_k,G) - 2$ which has an exponential lower bound, and so the $\Stab(\F)$-word length of $\phi_k$ has an exponential lower bound, completing the proof in Case~2.

\subparagraph{Case 3: Free factor systems with $\ge 3$ components.} We reduce to case 2 as follows. We wish to prove that $\Stab(\F)$ is distorted, where $\F=\{[A_0],[A_1],\ldots,[A_{M-1}]\}$ has $M \ge 3$ components. As in case~2, let $G'$ be a marked graph having rose subgraphs $H'_0,H'_1,H'_2$, an edge $\eta'_0$ from the vertex $v'_0 \in H'_0$ to the vertex $v'_1 \in H'_1$, and an edge $\eta'_1$ from the vertex $v'_2 \in H'_2$ to $v'_1$, so that the component $[A_0]$ of $\F$ is the free factor support of $H'_0$, the component $[A_1]$ is the free factor support of $H'_1$, but now we add the additional requirement that the remaining components $[A_2],\ldots,[A_{M-1}]$ of $\F$ are supported by subroses of $H'_2$. The construction of Case~2 produces a sequence $\phi_k \in \Out(F_n)$ which stabilizes all free factor systems $\A \sqsubset \{[H'_0 \union \eta'_0 \union H'_1],[H'_2]\}$; this includes both of the free factor systems $\{[A_0],[A_1]\}$ and $\F$. Also, the construction proves that the word length of $\phi_k$ in $\Out(F_n)$ has a linear upper bound and that its word length in $\Stab\{[A_0],[A_1]\}$ has an exponential lower bound. 

The group $\Stab(\F)$ acts on the finite set $\F$; let $\Stab_0(\F)$ be the kernel of this action, fixing each of the connected free factors $[A_0]$, $[A_1]$, \ldots, $[A_{M-1}]$. It follows that $\Stab_0(\F)$ is a subgroup of $\Stab\{[A_0],[A_1]\}$. Since the word length of $\phi_k$ in $\Stab\{[A_0],[A_1]\}$ has an exponential lower bound, its word length in $\Stab_0(\F)$ must also have an exponential lower bound. Since $\Stab_0(\F)$ is a finite index subgroup of $\Stab(\F)$, it follows that the word length of $\phi_k$ in $\Stab(\F)$ has an exponential lower bound, completing the proof in Case~3.


\begin{thebibliography}{ECH{\etalchar{+}}92}

\bibitem[Ali02]{Alibegovic:translation}
E.~Alibegovi{\'c}, \emph{Translation lengths in {${\rm Out}(F\sb n)$}},
  Geometriae Dedicata \textbf{92} (2002), 87--93.

\bibitem[BFH00]{BFH:TitsOne}
M.~Bestvina, M.~Feighn, and M.~Handel, \emph{{The Tits alternative for ${\rm
  Out}(F\sb n)$. I. Dynamics of exponentially-growing automorphisms.}}, Ann. of
  Math. \textbf{151} (2000), no.~2, 517--623.

\bibitem[Bur]{Burillo:private}
J.~Burillo, private correspondence.

\bibitem[BV95]{BridsonVogtmann:GeometryOfAut}
M.~Bridson and K.~Vogtmann, \emph{On the geometry of the automorphism group of
  a free group}, Bull. London Math. Soc. \textbf{27} (1995), no.~6, 544--552.

\bibitem[BV10]{BridsonVogtmann:DehnFunctions}
M.~R. Bridson and K.~Vogtmann, \emph{{The Dehn functions of $Out(F_n)$ and
  $Aut(F_n)$}}, Preprint arXiv:1011.1506v1, 2010.

\bibitem[CV86]{CullerVogtmann:moduli}
M.~Culler and K.~Vogtmann, \emph{Moduli of graphs and automorphisms of free
  groups}, Invent. Math. \textbf{84} (1986), 571--604.

\bibitem[ECH{\etalchar{+}}92]{ECHLPT}
D.B.A. Epstein, J.~Cannon, D.F. Holt, S.~Levy, M.S. Paterson, and W.P.
  Thurston, \emph{Word processing and group theory}, Jones and Bartlett, 1992.

\bibitem[GS91]{GerstenShort:rational}
S.~M. Gersten and H.~Short, \emph{Rational subgroups of biautomatic groups},
  Ann. of Math. \textbf{134} (1991), 125--158.

\bibitem[Ham05]{Hamenstadt:MCGII:quasigeodesics}
U.~Hamenstadt, \emph{{Geometry of the mapping class groups II:
  (Quasi)-geodesics}}, arXiv:math.GR/0511349, 2005.

\bibitem[Hat95]{Hatcher:HomStability}
A.~Hatcher, \emph{Homological stability for automorphism groups of free
  groups}, Comment. Math. Helv. \textbf{70} (1995), no.~1, 39--62.

\bibitem[Haw08]{Hawkins:PerronFrobenius}
T.~Hawkins, \emph{Continued fractions and the origins of the
  {P}erron-{F}robenius theorem}, Arch. Hist. Exact Sci. \textbf{62} (2008),
  no.~6, 655--717. \MR{MR2457065 (2009i:01018)}

\bibitem[HV96]{HatcherVogtmann:isoperimetric}
A.~Hatcher and K.~Vogtmann, \emph{Isoperimetric inequalities for automorphism
  groups of free groups}, Pac. J. Math. \textbf{173} (1996), no.~2, 425--441.

\bibitem[HV98]{HatcherVogtmann:Cerf}
\bysame, \emph{Cerf theory for graphs}, J. London Math. Soc. (2) \textbf{58}
  (1998), no.~3, 633--655.

\bibitem[Mil68]{Milnor:curvature}
J.~Milnor, \emph{A note on curvature and fundamental group}, J. Diff. Geom.
  \textbf{2} (1968), 1--7.

\bibitem[MM00]{MasurMinsky:complex2}
H.~Masur and Y.~Minsky, \emph{{Geometry of the complex of curves II:
  Hierarchical structure}}, Geom. Funct. Anal. \textbf{10} (2000), no.~4,
  902--974.

\bibitem[Ser80]{Serre:trees}
J.~P. Serre, \emph{Trees}, Springer, New York, 1980.

\bibitem[Spa81]{Spanier}
E.~Spanier, \emph{Algebraic topology}, Springer, New York---Berlin, 1981.

\bibitem[Vog02]{Vogtmann:OuterSpaceSurvey}
K.~Vogtmann, \emph{Automorphisms of free groups and outer space}, Proceedings
  of the Conference on Geometric and Combinatorial Group Theory, Part I (Haifa,
  2000), vol.~94, 2002, pp.~1--31.

\bibitem[\v{S}55]{Svarc:VolumeInvariant}
A.~S. \v{S}varc, \emph{A volume invariant of coverings}, Dokl. Akad. Nauk SSSR
  \textbf{105} (1955), 32--34.

\end{thebibliography}

\newcommand{\etalchar}[1]{$^{#1}$}
\def\cprime{$'$} \def\cprime{$'$}
\providecommand{\bysame}{\leavevmode\hbox to3em{\hrulefill}\thinspace}
\providecommand{\MR}{\relax\ifhmode\unskip\space\fi MR }
\providecommand{\MRhref}[2]{%
  \href{http://www.ams.org/mathscinet-getitem?mr=#1}{#2}
}
\providecommand{\href}[2]{#2}

 \end{document}